\documentclass[oneside,reqno]{amsart}
\usepackage{bbm}
\usepackage{mathabx}
\usepackage{tensor}
\usepackage{amssymb}
\usepackage{amsmath}
\usepackage{amsfonts}
\usepackage{amsthm}
\usepackage{stmaryrd}
\usepackage[all]{xy}
\usepackage{mathrsfs}
\usepackage{graphicx}
\usepackage{hyperref}
\usepackage{color}
\usepackage{tikz-cd}
\usepackage{tkz-euclide}
\usepackage{breqn}
\usepackage{tikz}
\usetikzlibrary{matrix}
\usepackage{enumitem}
\usepackage{amsmath,amscd}
\usepackage{pifont}
\usetikzlibrary{arrows}
\usepackage[all]{xy}
\usepackage{graphicx}
\usepackage{placeins}
\usepackage{xspace}
\usetikzlibrary{calc,intersections,through,backgrounds}

\newtheorem{thm}{Theorem}[section]
\newtheorem{theorem}[thm]{Theorem}

\newtheorem{corollary}[thm]{Corollary}
\newtheorem{lemma}[thm]{Lemma}
\newtheorem{proposition}[thm]{Proposition}

\theoremstyle{definition}

\newtheorem{definition}[thm]{Definition}
\newtheorem{example}[thm]{Example}

\newtheorem{notation}[thm]{Notation}

\newtheorem{remark}[thm]{Remark}
\newtheorem{property}[thm]{Property}

\newtheorem{hypothesis}[thm]{Hypothesis}

\numberwithin{equation}{section}

\numberwithin{mytheorem}{subsection}

\numberwithin{mytheorem}{subsection}
\numberwithin{myconjecture}{subsection}
\numberwithin{mydefinition}{subsection}
\numberwithin{myremark}{subsection}
\numberwithin{mysituation}{subsection}
\numberwithin{myhypothesis}{subsection}
\numberwithin{myquestion}{subsection}
\numberwithin{mynotation}{subsection}
\numberwithin{myfact}{subsection}
\numberwithin{myexamples}{subsection}
\numberwithin{myexample}{subsection}
\numberwithin{myconstruction}{subsection}
\numberwithin{mycaution}{subsection}
\numberwithin{myproposition}{subsection}
\numberwithin{mylemma}{subsection}
\numberwithin{mycorollary}{subsection}

\def\AAA{\mathbb{A}}

\def\CC{\mathbb{C}}

\def\calF{\mathbb{E}}
\def\FF{\mathbb{F}}
\def\GG{\mathbb{G}}

\def\LL{\mathbb{L}}
\def\MM{\mathbb{M}}
\def\NN{\mathbb{N}}

\def\QQ{\mathbb{Q}}
\def\RR{\mathbb{R}}
\def\SS{\mathbb{S}}

\def\WW{\mathbb{W}}

\def\ZZ{\mathbb{Z}}

\def\calF{\mathcal{F}}

\def\calO{\mathcal{O}}

\def\sp{\mathrm{sp}}

\def\HP{\mathrm{HP}}
\def\NP{\mathrm{NP}}
\def\IHP{\mathrm{IHP}}
\def\GNP{\mathrm{GNP}}

\def\sgn{\mathrm{sgn}}
\def\val{\mathrm{val}}

\def\diag{\mathrm{diag}}

\def\Vol{\mathrm{Vol}}
\def\Deg{\mathrm{Deg}}
\def\Zar{\mathrm{Zar}}
\def\res{\mathrm{res}}

\def\LD{\mathrm{LD}}
\def\Ver{\mathrm{Ver}}

\def\Frob{\mathrm{Frob}}

\def\Zp{\ZZ_p}
\def\ux{\underline{x}}

\newcommand{\Tr}{\mathrm{Tr}}

\DeclareMathOperator{\Cone}{Cone}

\DeclareMathOperator{\Iso}{Iso}

\pgfkeys{tikz/mymatrixenv/.style={decoration=brace,every left delimiter/.style={xshift=3pt},every right delimiter/.style={xshift=-3pt}}}
\pgfkeys{tikz/mymatrix/.style={matrix of math nodes,left delimiter=[,right delimiter={]},inner sep=2pt,column sep=1em,row sep=0.5em,nodes={inner sep=0pt}}}
\pgfkeys{tikz/mymatrixbrace/.style={decorate,thick}}

\begin{document}\large

\title[Generic Newton polygon for exponential sums with parallelotope base] {Generic Newton polygon for exponential sums in $n$ variables with parallelotope base}

\author{Rufei Ren}
\address{Department of Mathematics 
	University of Rochester 
	915 Hylan Building,
	Rochester, NY 14627}
\email{rren2@ur.rochester.edu}
\date{\today}
\begin{abstract}
	Let $p$ be a prime number. Every $n$-variable polynomial $f(\underline x)$ over a finite field of characteristic $p$ defines an Artin--Schreier--Witt tower of varieties whose Galois group is isomorphic to $\ZZ_p$. 
	Our goal of this paper is to study the Newton polygon of the $L$-function associated to a nontrivial finite character of $\ZZ_p$ and a generic polynomial whose convex hull is an $n$-dimensional paralleltope $\Delta$. We denote this polygon by $\GNP(\Delta)$. We prove a lower bound of $\GNP(\Delta)$, 
	which is called the improved Hodge polygon $\IHP(\Delta)$. We show that $\IHP(\Delta)$ lies above the usual Hodge polygon $\HP(\Delta)$ at certain infinitely many points, and when $p$ is larger than a fixed number determined by $\Delta$, it coincides with  $\GNP(\Delta)$ at these points. As a corollary, we roughly determine the distribution of the slopes of $\GNP(\Delta)$.
\end{abstract}

	\subjclass[2010]{11T23 (primary), 11L07 11F33 13F35 (secondary).}
	\keywords{Artin--Schreier--Witt towers, $T$-adic exponential sums, Slopes of Newton polygon, $T$-adic Newton polygon for Artin--Schreier--Witt towers, Eigencurves}
	\maketitle
	
	\setcounter{tocdepth}{1}
	\tableofcontents
	
	\section{Introduction}\label{section 1} 
	We shall state our main results and their motivation after recalling the notion of
	$L$-functions for Witt coverings. 
	Let $p$ be a prime number, and
	 $$f(\ux):=\sum_{P\in \ZZ^n_{\geq 0}} a_{P}\ux^{P}$$ be an $n$-variable polynomial in $\overline\FF_p[x_1,\dots,x_n]$. We denote by $\FF_p(f)$ the coefficient field of $f(\ux)$ and set $m(f):=[\FF_p(f):\FF_p]$.  
	 Then we put $\hat a_P \in \ZZ_{p^{m(f)}}$ to be the Teichmuller lift of $a_P$, where $\ZZ_{p^{m(f)}}$ is the unramified extension of $\ZZ_p$ of degree $m(f)$, and call $$\hat{f}(\ux):=\sum_{P\in \ZZ^n_{\geq 0}} \hat a_{P} \ux^{P}$$ the \emph{Teichm\"uller lift} of $f(\ux)$.
	 
	Viewing $\ZZ^n$ as the lattice points in $\RR^n$ with origin denoted by $\calO$, we call the convex hull of $\calO\cup\big\{P\;|\;a_{P}\neq 0\big\}$ \emph{the polytope} of $f$ and denote it by $\Delta_f$.
	
	The \emph{Artin--Schreier--Witt tower}
	associated to $f$ is the sequence of varieties $\mathcal V_i$ over $\FF_{p^{m(f)}}$ defined by the following equations: 
	\[\mathcal V_i:y_i^F-y_i=\sum_{P\in \Delta_f}(a_{P}{\ux}^P,0,0,\dots)_i,\]
 	where $y_i=(y_i^{(1)}, y_i^{(2)},\dots, y_i^{(i)})$ are Witt vectors of length $i$, and $\bullet^F$ means raising
	each Witt coordinate to the $p$th power. 
	The Artin--Schreier--Witt tower $\cdots\to \mathcal V_i\to\cdots\to \mathcal V_0:=\AAA^n$
	is a tower of Galois covers of $\AAA^n$
	with total Galois
	group $\Zp$, and consequently the study of zeta function of the tower can be reduced to the study of the L-functions associated to (finite) characters of the Galois group $\Zp$. For more details we refer the readers to \cite[\S 1]{Davis-wan-xiao}.
	
	Let $(\GG_{m})^n$ be the $n$-dimensional torus over $\mathbb{F}_{p^{m(f)}}$.
	The main subject of our study is the $L$-function associated to a finite character $\chi: \Zp \to \CC_p^\times$ of conductor $p^{m_\chi}$ which is given by
	$$L_f^*(\chi,s): =\prod\limits_{\ux\in |(\GG_{m})^n|}
	\frac{1}{1-\chi\Big(\Tr_{\QQ_{p^{m(f)\deg(\ux)}}/\QQ_p}(\hat{f}(\hat{\ux}))\Big)s^{\deg(\ux)}},$$
		where $|(\GG_{m})^n|$ is the set of closed points of $(\GG_{m})^n$, $\hat\ux$ is the Teichmuller lift of a geometric point at $\ux$, and $\deg(\ux)$ stands for the degree of $\ux$ over $\FF_{p^{m(f)}}$.

			It is proved in \cite[Theorem~1.3]{liu-wei} that when $f$ is a non-degenerate polynomial with convex hull $\Delta_f$, the function $$L^*_f(\chi, s)^{{(-1)}^{n-1}}=\sum_{i=0}^{n!p^{n(m_\chi-1)} \Vol (\Delta_f)} v_{i}s^i \in \ZZ_p[\zeta_{p^{m_\chi}}][s]$$ is a polynomial of degree $n!p^{n(m_\chi-1)} \Vol (\Delta_f)$, where $\zeta_{p^{m_{\chi}}}$ is a primitive $p^{m_{\chi}}$-th root of unity. 
			
			In this paper, we confine ourselves to studying these $f$ whose polytopes are of the simpler shape. 
			Let $\Delta$ be an $n$-dimensional paralleltope generated by linearly independent vectors $\overrightarrow{\calO \mathbf V_1},\overrightarrow{\calO \mathbf V_2},\dots,\overrightarrow{\calO \mathbf V_n}$, where $\mathbf V_1,\mathbf V_2,\dots,\mathbf V_n$ are integral points.
			
			\begin{hypothesis}\label{hypothesis}
			We assume that $p$ is a prime such that $p\nmid (n!\Vol(\Delta))$ and $p>(n+4)D$, where 
			$D$ is a positive integer depending only on $\Delta$ (See Definition~\ref{definition of D}).
		
			In particular, when $\Delta$ is an $n$-dimensional cube with side length $d$, the integer $D = d$.
			\end{hypothesis}
			 
			\begin{notation}\label{??}
		 Let $k\Delta$ denote a scaling of $\Delta$. We write $$\Delta_k^+: = k\Delta \cap \ZZ^n$$
		 for the set consisting of the lattice points in $k\Delta$ and put $$\mathbbm{x}^+_k:=\#\Delta^+_k.$$
		 Let $\Delta^\circ$ denote the paralleltope $\Delta$ with all faces not containing $\calO$ removed. We put $$\Delta_k^-: = k\Delta^\circ \cap \ZZ^n\quad \textrm{and}\quad \mathbbm{x}^-_k:=\#\Delta_k^-.$$ 
		 For simplicity, we write $\Delta^\pm$ for $\Delta^\pm_1$ when no confusion can rise. 
		\end{notation}
		
		%Moreover, for those simplex $$\Delta:=\left\{\theta_{0}u_{0}+\dots +\theta _{n}u_{n}\bigg|\theta_{i}\geq 0,0\leq i\leq n,\sum_{i=0}^{n}\theta_{i}=1\right\},$$ where $u_i\in \ZZ^n$, we denote $\square_{\Delta}:=\left\{\theta _{0}u_{0}+\dots +\theta _{n}u_{n}|0\leq\theta _{i}\leq 1\right\}.$
		
		\begin{notation}\label{notation for xi}
		 One may naturally identify the set of polynomials with polytope $\Delta$ as an open subscheme of  $\overline \FF_p^{\Delta^+}$ by recording the coefficients $a_P$ of $f$.
%			\begin{enumerate}
%				\item Let $$\calF(\Delta):=\Big\{ f( x)\in \overline \FF_p[ x]\;\big|\; \Delta_f=\Delta \Big\}.$$
%				\item 	We put 
%				\begin{equation}
%				\begin{split}
%				\xi:& \quad\quad\calF(\Delta)\ \ \longrightarrow \quad\overline \FF_p^{\mathbbm{x}_1^+}\\
%				&	\sum\limits_{i=1}^{\mathbbm{x}_1^+}a_{\mathbf{P}_i}{x}^{\mathbf{P}_i}	\longmapsto ( a_{\mathbf{P}_1},a_{\mathbf{P}_2},\dots, a_{\mathbf{P}_{\mathbbm{x}_1^+}}).
%				\end{split}
%				\end{equation}
%			\end{enumerate}

		\end{notation}
		
\begin{definition}\label{definition for NP chi} 
	If $f$ satisfies the non-degenerate condition in \cite{liu-wei}, we call the lower convex hull of the set of points $\Big(i,p^{m_\chi-1}(p-1)\val_{p^{m(f)}}(v_{i})\Big)$ the \emph{normalized Newton polygon of $L^*_f(\chi, s)^{{(-1)}^{n-1}}$} which is denoted by $\NP(f, \chi)_{L}$. Here, $\val_{p^{m(f)}}(-)$ is the $p$-adic valuation normalized so that $\val_{p^{m(f)}}(p^{m(f)})=1.$  
\end{definition}
		
			The following Theorem~\ref{theorem for L} and Theorem~\ref{generic newton polygon} are the main results of this paper.

		\begin{theorem}\label{theorem for L}
			Assume Hypothesis~\ref{hypothesis}. Let $\calF(\Delta)$ denote the set of all non-degenerate polynomials
			$f (\ux)= \sum\limits_{P \in \Delta^+} a_P  \ux^P\in \overline \FF_p[\ux]$ with $\Delta_f =\Delta$. 
			Then there is a Zariski open subset $O_{\Zar}\subset \calF(\Delta)$ such that for any $f\in O_{\Zar}$ and any finite character $\chi:\Zp\to \CC_p^\times$ of conductor $p^{m_\chi}$,
			if we put  $\left\{\alpha_1,\dots, \alpha_{p^{n(m_\chi-1)}n!\Vol(\Delta)}\right\}$ to be the set of $p^{m(f)}$-adic Newton slopes of 
			$L^*_f(\chi,s)^{{(-1)}^{n-1}}$, it has the following distribution.
			 
			 For every $0\leq i_1\leq n-1$ and every $0\leq i_2\leq p^{m_\chi-1}-1$ we have
			\begin{enumerate}
			\item $\#\Big\{\alpha_j\;\Big|\;\alpha_{j}\in (i_1+\frac{i_2}{p^{m_\chi-1}},i_1+\frac{i_2+1}{p^{m_\chi-1}})\Big\}=\sum\limits_{t=0}^{i_1}(-1)^t\binom{n}{t}\Big(\mathbbm x^-_{(i_1-t)p^{m_\chi-1}+i_2+1}-\mathbbm x^+_{(i_1-t)p^{m_\chi-1}+i_2}\Big),$
		\item	$\#\Big\{\alpha_j\;\Big|\;\alpha_{j}=i_1+\frac{i_2}{p^{m_\chi-1}}\Big\}= \sum\limits_{t=0}^{i_1}(-1)^t\binom{n}{t}\Big(\mathbbm x^+_{(i_1-t)p^{m_\chi-1}+i_2}-\mathbbm x^-_{(i_1-t)p^{m_\chi-1}+i_2}\Big).$\\
			\end{enumerate}

		\end{theorem}
	\begin{remark}
		\begin{itemize}
			\item By Poincar\'e duality, all slopes of $L^*_f(\chi,s)^{{(-1)}^{n-1}}$ belong to $[0,n)$.
		\item The distribution above includes all slopes of $L^*_f(\chi,s)^{{(-1)}^{n-1}}$. 
		\end{itemize}
	\end{remark}
	
			To study $L$-function it is more convenient to work with the so-called \emph{characteristic power series} 
		\begin{equation}\label{C(chi, s)}
		C^*_f(\chi, s):=\Big(\prod\limits_{j=0}^{\infty}L^*_f(\chi, {p^{m(f)j}}s)^{ \binom{n+j-1}{n-1}}\Big)^{{(-1)}^{n-1}},
		\end{equation}
		whose normalized $p$-adic Newton polygon we denote by $\NP(f, \chi)_{C}$. 
		
		\begin{definition}\label{generic Newton polygon} 
		The \emph{generic Newton polygon} of $\Delta$ is defined by
			$$\GNP(\Delta):=\inf\limits_{\substack{\chi: \Zp/p^{m_\chi}\Zp\to \CC_p^\times\\\Delta_f=\Delta}}\Big(\NP(f, \chi)_{C}\Big),$$
			where $\chi:\ZZ_p\to \CC_p^\times$ runs over all nontrivial finite characters, and $f$ runs over all non-degenerate polynomials in $\overline \FF_p[ x]$ such that $\Delta_f=\Delta$. 
		\end{definition}
	
%		Similarly,
%		we write $\NP(f, \chi)_C$ for the normalized Newton polygon of $C^*_f(\chi, s)$.

		\begin{theorem}\label{generic newton polygon}
		 The generic Newton polygon $\GNP(\Delta)$ passes through the points	
			$(\mathbbm{x}_k^\pm(\Delta), h(\Delta^\pm_k))$ for any $k\geq 0$, where 
			$\mathbbm{x}^\pm_k$ and $h(\Delta^\pm_k)$ are defined in Notation~\ref{cone} and Definition~\ref{definition of h} respectively.
		\end{theorem}

		In \cite{Davis-wan-xiao}, Davis, Wan, and Xiao studied the $p$-adic Newton slopes  of $L^*_f(\chi, s)$ and $C^*_f(\chi, s)$  when $f$ is a one-variable polynomial whose degree $d$ is coprime to $p$. They concluded that, for each character $\chi:\Zp\to \CC_p^\times$ of a relatively large conductor, the $p$-adic Newton slopes of $L^*_f(\chi, s)$ are in a finite union of arithmetic progressions.
		Their proof strongly inspired the proof of spectral halo conjecture by Liu, Wan, and Xiao in \cite{liu-wan-xiao}; we refer to \cite[\S 1.5]{ren-wan-xiao-yu} for the discussion on the analogy of the two proofs. 
		
		Motivated by the attempt of extending spectral halo type results beyond the case of modular forms, 
		it is natural to ask whether one can generalize the main results of \cite{Davis-wan-xiao} to more general cases: 
		\begin{enumerate}
			\item changing the tower to $\ZZ_{p^\ell}$ for $\ell\geq2$, and
			\item making the base to higher dimensional.
		\end{enumerate}
	
		The first case is examined in a joint work with Wan, Xiao, and Yu see \cite{ren-wan-xiao-yu}. The goal of this current paper is to investigate the second case.
		
		From the Iwasawa theory point of view, it is important to have access to the Newton polygon $\NP(f,\chi)_C$ associated to this Artin-Schreier-Witt tower.  When $p$ is ``ordinary", this polygon was explicitly computed by Adolphson--Sperber \cite{AS}, Berndt--Evans \cite{BR}, and Wan \cite{wan} in many special cases, and by Liu-Wan \cite{liu-wan} in the general case (and in the $T$-adic setup).

%	
%	 Newton polygon of $\NP(f, \chi)_{C}$ plays a key role in studying the Artin--Schreier--Witt tower associated to $f$. We introduce a little bit about its development.

	Going beyond the ordinary case, there has been many researches on understanding the generic Newton polygon of $L_f(\chi, s)$ when $f$ is a polynomial of a single variable. Here is an incomplete list.
	\begin{itemize}
		\item In \cite{Davis-wan-xiao}, Davis, Wan, and Xiao prove that the Newton slopes of $L_f(\chi,s)$ form a finite union of arithmetic progressions, when $f$ is a one-variable polynomial and $\chi$ is a finite character of a relatively large conductor. 
		\item When $p$ is large enough, Zhu \cite{zhu03} and Scholten--Zhu \cite{sz} showed that for a non-degenerate one-variable polynomial $f$ and a finite character $\chi_0$ of conductor $p$, the Newton polygon $\NP(f,\chi_0)_L$ coincides $\GNP(\Delta)$.
		\item Later, Blache, Ferard, and Zhu in \cite{bfz} proved a lower bound for the Newton polygon of $f(x)\in\FF_q[x,\frac{1}{x}]$ of degree $(d_1, d_2)$, which is called a Hodge-Stickelberger polygon. They also showed that when $p$ approaches to infinity, the Newton polygon $\NP(\chi,f)_L$ coincides with the  Hodge-Stickelberger polygon. 
		\item  In \cite{bf}, Blache and Ferard worked on the generic Newton polygon associated to characters of large conductors.
		\item In \cite{ouy}, Ouyang and Yang studied the one-variable polynomial $f(x) = x^d +a_1x$. A similar result can be found in \cite{ouz}, where Ouyang and Zhang studied the family of polynomials of the form  $f(x) = x^d +a_{d-1}x^{d-1}$. 	
		\item In \cite{kw}, Koster and Wan studied a more general $\ZZ_p$-tower, and they proved the genus stability of all such $\ZZ_p$-tower.
	\end{itemize}
	However, for technical reasons, it is difficult to prove that the slopes of the $L$-function form a union of arithmetic progression when $f$ is a multi-variable polynomial. 
	Zhu in \cite{zhu} shows that  $\GNP(\Delta)$ and $\IHP(\Delta)$ coincide for characters of $\ZZ_p$ of conductor $p$ when $\Delta$ is a rectangular and $p$ is large enough. 
	A similar result is obtained by the author in \cite{Ren2} when the polytope of $f$ is an isosceles right triangle. 
	
	In this paper, we focus on the generic Newton polygon of an $n$-dimensional parallelotope $\Delta$.  
	Our main contribution in this paper is to prove the distribution of slopes of the generic Newton polygon $\GNP(\Delta)$, when  $p$ is not necessary to be ordinary with respect to $\Delta$. We refer reader to Theorem~\ref{theorem for L} for the statement.

%	in this paper is to find an improved lower bound $\IHP(\Delta)$ for $\NP(f, \chi)_{C}$ which is strictly above the classical Hodge bound when $\Delta$ is non-ordinary, and
%	prove that $\IHP(\Delta)$ coincides with $\GNP(\Delta)$ at infinitely points. As a corollary, we deduce that the slopes of $\GNP(\Delta)$ roughly form an arithmetic progression with increasing multiplicities. 
%	
%	The proof of  falls into four steps.
	\medskip
	Now we list the key steps of the proof of our main theorems.
	
	{Step 1}: Instead of working with the $L$-function itself, it is more convenient (from the point of view of Dwork trace formula) to work with the characteristic power series $C^*_f(\chi, s)$, which recovers the $L$-function by 
	\begin{equation}\label{equation LC}
	L^*_f(\chi, s)^{(-1)^{n-1}}=\prod\limits_{j=0}^{n}C^*_f(\chi, {p^{m(f)j}}s)^{(-1)^j\binom{n}{j}}.
	\end{equation}
%	So it suffices to study the Newton polygon $\NP(f, \chi)_{C}$. 
	
	The power series $C^*_f(\chi, s)$ is genuinely the characteristic power series of a nuclear operator (or equivalently an infinite matrix $N$ with respect to some canonical basis). Moreover, we can do this for the universal character (as opposed to just finite characters) of the Galois group of the tower $\ZZ_p$.

	{Step 2}: We construct the improved Hodge polygon $\IHP(\Delta)$ for $\Delta$ in Definition~\ref{IHP}, and prove that it is a lower bound of $\NP(f, \chi)_{C}$ for every nontrivial finite character $\chi$. 
The polygon $\IHP(\Delta)$ lies above the usual Hodge polygon at $x=\mathbbm x_k^\pm$ for every $k\geq1$. In fact, we show in Proposition~\ref{IHPHP} the condition that $\IHP(\Delta)$ lies strictly above the usual Hodge polygon at $x=\mathbbm x_k^\pm$.

%when $p$ is non-ordinary with respect to $\Delta$. 
	
	The key point of this step lies in: the usual way of obtaining Hodge polygon is to conjugate $N$ by an appropriate diagonal matrix, and observe that each row is entirely divisible by a certain power of $p$. In this paper, we dig into the definition of characteristic power series as the sum over permutations, which allows us to slightly but crucially improve the usual Hodge polygon.

 {Step 3}: We show in Proposition~\ref{proposition 1} that for every polynomial $f\in \mathcal F(\Delta)$, if there is a finite character $\chi_0$ of conductor $p$ such that $\NP(f, \chi_0)_{C}$ coincides with $\IHP(\Delta)$ at $$x=\mathbbm x^\pm_k\quad \textrm{for}\quad 1\leq k\leq n+2,$$ then for every nontrivial finite character $\chi$, $\NP(f, \chi)_{C}$ and $\IHP(\Delta)$ coincide at $$x=\mathbbm x^\pm_k \quad \textrm{for all}\quad k\geq 0.$$

This proposition reduces the problem to show that all the polynomials $f\in \mathcal F(\Delta)$ such that $\NP(f, \chi_0)_{C}$ and $\IHP(\Delta)$ coincide at $x=\mathbbm x^\pm_k$ for $1\leq k\leq n+2$ form a Zariski open dense subset of $\mathcal F(\Delta).$
%	Obviously, this theorem is very powerful since its condition is only on a special character such that  $\NP(f, \chi_0)_{C}$ and $\IHP(\Delta)$ coincides at finite points, but its gaurentees $\NP(f, \chi)_{C}$ and $\IHP(\Delta)$ rouches an arbitrary nontrivial finite character 

For this, one may consider the characteristic power series for a ``universal'' polynomial $\widetilde{f}$, namely all coefficients of $\widetilde{f}$ are viewed as variables. We need to show that when we write $\widetilde u_{\mathbbm x_k^\pm}=\widetilde u_{\mathbbm x_k^\pm,h(\Delta_k^\pm)}T^{h(\Delta_k^\pm)}+O(T^{h(\Delta_k^\pm)+1}),$ the coefficients
\begin{equation}\label{1.4}
\widetilde u_{\mathbbm x_k^\pm,h(\Delta_k^\pm)}\not\equiv 0 \pmod{p}\quad \textrm{	for every} \quad 1\leq k\leq n+2.
\end{equation}

{Step 4}: We show \eqref{1.4} in \S\ref{section 5}.
The technical core of this paper lies in proving \eqref{1.4}. Roughly speaking,  the key is to show that for each $0\leq k\leq n+2$, a certain monomial of $\widetilde u_{\mathbbm x_k^\pm}$ is nonzero. Tracing back to the definition of $\widetilde u_{\mathbbm x_k^\pm}$, we see the contribution to such \emph{leading} monomial must come from a unique special permutation that appears in the definition of the characteristic power series. Computing explicitly the contribution of this special permutation to the leading term, which itself was subdivided into simpler cases, allows us to prove \eqref{1.4}.

%	
%	
%	Recently, a lot of meaningful and interesting results about the Artin--Schreier--Witt towers and their analogy come out. For example:
%	In \cite{kw}, Wan and Koster proved a genus bound of an
%	Artin-Schreier-Witt towers. In \cite{hae}, Haessig provided estimates for the $p$-adic valuations of the roots of $L$-functions associated to symmetric powers of Kloosterman sums; more recently, Zhu stated a similar result to this paper in \cite{zhu} when the polytope of $f$ is a rectangular. 

	\subsection*{Acknowledgments}
	The author would like to thank his advisor Liang Xiao for the many help and support on this paper. The author is grateful for the anonymous referee to many useful comments to improve the presentation. He also thanks Douglas Haessig, Hui June Zhu, and Daqing Wan for helpful discussions.

\section{Dwork's trace formula}\label{section 2} 
In this section, we will introduce Dwork trace formula to express $C^*_f(\chi,s)$ as the characteristic power series of some infinite matrix and deduce a natural lower bound of $\NP(T,f)_C$ called the improved Hodge polygon.

Recall from introduction that $\Delta$ is an $n$-dimensional paralleltope generated by linearly independent vectors $\overrightarrow{\calO \mathbf V_1},\overrightarrow{\calO \mathbf V_2},\dots,\overrightarrow{\calO \mathbf V_n}$.
\begin{notation}\label{cone}
	We denote the \emph{cone} of $\Delta$ by 
	$$\Cone(\Delta):=\Big\{Q\in \RR^n\;\big |\; kQ\in \Delta\ \textrm{for some}\ k>0 \Big\},$$ and put $$\MM(\Delta):=\Cone(\Delta)\cap \ZZ^n$$ to be the set of integer point points in $\Cone(\Delta)$.
	
	\end{notation}

\begin{definition}\label{definition of D}
	Let $D$ be the smallest positive integer such that \begin{equation}\label{equation w}
	\MM(\Delta)\subset\Big\{z_1 \mathbf{V}_1+z_2\mathbf{V}_2+\cdots+z_n\mathbf{V}_n\;\Big|\; z_i\in \frac{1}{D}\ZZ_{\geq 0}\Big\}.
	\end{equation}
	In particular, when $\Delta$ is an $n$-cube with side length $d$, we have $D = d$.
\end{definition}
%	It is easy to check that
%$D\;\big|\; \Vol(\Delta).$

\begin{notation}
	Let $$\Lambda_{\Delta}:=\bigoplus_{i=1}^n  \ZZ_{\geq 0}\cdot\mathbf V_i\subset \MM(\Delta).$$
\end{notation}

%	For $j=1,\cdots,n$ and a face $\tau$ of $\Delta$ not containing $\calO$, we write
%$$\overline{_jf}^{\tau}=\sum
%\limits_{i=0}^{m-1}\sum\limits_{p^{m-i-1}P\in\tau}P_ja_{iu}^{p^{m-i-1}}x^{p^{m-i-1}P},$$
%where $u_j$ is the $j$-th coordinate of $u$. 
%
%\begin{definition}\label{nondegenerate}
%	We call $f$
%non-degenerate with respect to $\Delta_f$ if
%$\Delta_f$ is of dimension $n$, and for every face
%$\tau$ of $\Delta_f$ that does not contain $\calO$, the system
%$\overline{_1f}^{\tau} =\cdots=\overline{_nf}^{\tau}$ has no
%common solution in $(\overline{\mathbb{F}}_q^{\times})^n$. We refer readers to \cite[\S 1]{liu-wei} for more details. 
%\end{definition}

%Recall that $\Delta^\pm_k$, in Notation~\ref{??}, are subsets of $\MM(\Delta)$. 
	
From now on, let $p$ be a prime satisfying Hypothesis~\ref{hypothesis}, and let $\calF(\Delta)$ denote the set of all non-degenerate polynomials
$f (\ux)= \sum\limits_{P \in \Delta^+} a_P  \ux^P\in \overline \FF_p[\ux]$ with $\Delta_f =\Delta$. We denote by $\FF_p(f)$ the \emph{coefficient field} of $f$ which is the finite field generated by the coefficients of $f$.

We will fix such a polynomial $f\in \calF(\Delta)$ in \S\ref{section 2} and \S\ref{section 3}. We set $\FF_q:=\FF_p(f)$ and $m=[\FF_q:\FF_p]$. 
Let $\hat a_{P} \in \ZZ_q$ be the Teichm\"uller lift of $a_{P}$. We call $\hat f(\ux):=\sum\limits_{P\in \Delta^+} \hat a_{P} \ux^P$ the \emph{Teichm\"uller lift} of $f(\ux)$.

%All the results for $f$ can be applied analogously to any polynomial of convex hull $\Delta_f$ and finite coefficient field.

%We normalized the $p$-adic valutions $\val_p(-)$ (resp. $\val_{q}(-)$) so that $\val_{p}(p)=1$ (resp. $\val_{q}(q)=1$).

%Moreover, for those simplex $$\Delta:=\left\{\theta_{0}u_{0}+\dots +\theta _{n}u_{n}\bigg|\theta_{i}\geq 0,0\leq i\leq n,\sum_{i=0}^{n}\theta_{i}=1\right\},$$ where $u_i\in \ZZ^n$, we denote $\square_{\Delta}:=\left\{\theta _{0}u_{0}+\dots +\theta _{n}u_{n}|0\leq\theta _{i}\leq 1\right\}.$

\subsection{$T$-adic exponential sums.}
%We fix the polynomial $\overline f$ and its Teichm\"uller lift $f$ as in the introduction.

\begin{notation}\label{c}
	We recall that the \emph{Artin--Hasse exponential series} is defined by
	\begin{equation}\label{Artin-Hasse}
	\begin{split}
		E(\pi):=\sum_{i=0}^{\infty}c_i\pi^i = \exp\big( \sum_{i=0}^\infty \frac{\pi^{p^i}}{p^i} \big)
		 \in 1+ \pi + \pi^2 \ZZ_p[\![ \pi ]\!].
	\end{split}
	\end{equation}
	
	Setting $ E(\pi)= 1+T$ gives an isomorphism $\ZZ_p\llbracket\pi \rrbracket \cong \ZZ_p\llbracket T\rrbracket$. 
\end{notation}

	\begin{definition}\label{T-adic valuation}
		For a ring $R$ and a power series $g \in R\llbracket T\rrbracket$, we define its \emph{$T$-adic valuation}, denoted by $\val_T(g)$, as the largest integer $k$ such that $g\in T^kR\llbracket T\rrbracket$.
	\end{definition}

%	For the technical reason, since there is no dif and only iference for $T$-adic valuation and $\pi$-adic valuation, we sometimes do not distinguish them.

\begin{definition}	
	For every $k\geq 1$, the \emph{$T$-adic exponential sum} of $f$ over $\FF_{q^k}^\times$ is
	\[
	S_f^*(k, T): = \sum_{\ux\in (\FF_{q^k}^\times)^n} (1+T)^{\Tr_{\QQ_{q^k} / \QQ_p}(\hat{f}(\hat{\ux}))} \in \ZZ_p[\![T]\!].
	\]
	\begin{definition}
	The \emph{$T$-adic characteristic power series} of $f$ is defined by 
	\begin{equation}
		\label{E:Cfstar}
		C_f^*(T,s) := \exp \Big( \sum_{k=1}^\infty -(q^k-1)^{-n} S_f^*(k,T)\frac{s^k}{k} \Big)=
\sum_{k=0}^\infty u_{k}( T) s^k \in \ZZ_p\llbracket T, s\rrbracket.
	\end{equation}

			\end{definition}
	
	It is not difficult to check that every nontrivial finite character $\chi: \ZZ_{p} \to \CC_p^\times$ satisfies
	\begin{equation}\label{specialization}
		C_f^*(\chi,s) =C_f^*(T,s)\big|_{T = \chi(1)-1},
	\end{equation}
where $C_f^*(\chi,s)$ is defined in \S\ref{section 1}. We refer the readers to \cite[\S 2]{Davis-wan-xiao} for the proof.
	
\end{definition}

\begin{notation}
 We put
	\begin{equation}
	E_f(\ux):= \prod\limits_{P\in \Delta^+} E(\hat{a}_{P} \pi  \ux^P)
	 \in \ZZ_q\llbracket T\rrbracket \llbracket \ux \rrbracket
	\end{equation}
 and 
		\begin{equation}	\label{E:Ef(x)}
	\prod\limits_{P\in \Delta^+\backslash \{\calO\}} E(\hat{a}_{P} \pi  \ux^P)
	=\sum\limits_{Q\in \ZZ^n_{\geq 0}} e_{Q}(T) \ux^Q\in \ZZ_q\llbracket T\rrbracket \llbracket \ux \rrbracket.
		\end{equation}
	
	We shall later in \S\ref{section 4} and \S\ref{section 5} need a version of $E_f(\ux)$ for a universal polynomial $\widetilde f(\ux)$. Namely, we consider the universal polynomial $$\widetilde{f}(\ux)=\sum_{P\in \Delta^+}\widetilde a_{P} \ux^P\in \FF_p[\tilde a_P; P \in \Delta^+][\ux],$$ where $\widetilde a_{P}$ are treated as variables. Then we put
	\begin{equation}	\label{equation: E Delta}
	\begin{split}
	\prod\limits_{P\in \Delta^+\backslash \{\calO\}} E(\widetilde a_{P} \pi  \ux^P)
	=\sum\limits_{Q\in \ZZ^n_{\geq 0}} \widetilde e_{Q}(T) \ux^Q\in \ZZ_p[\widetilde{a}_P; P\in \Delta^+\backslash \{\calO\}] \llbracket T\rrbracket \llbracket \ux \rrbracket.
	\end{split}
	\end{equation}	
%	If $\tau$ denotes arithmetic $p$-Frobenius automorphism which acts naturally on $\QQ_q$, and  trivially on $\pi$ and $x$, then we have,  for every $j \in \ZZ_{\geq 0}$, 
%	\[
%	E_f^{\tau^j}(x)_\pi = \prod\limits_{i=0}^d E(a_i^{\tau^j} \pi x^i) \in \ZZ_q[\![T]\!] [\![ x ]\!].\]
\end{notation}

%\begin{convention}
%	\label{Conv:matrices start with zero}
%	In this paper, the row and column indices of matrices start with zero.
%\end{convention}

\subsection{Dwork's trace formula}\label{section 2.2}

\begin{definition}\label{Banach space}
We fix a $D$-th root $T^{1/D}$ of $T.$ Define
$$\textbf{B}=\Big\{\sum\limits_{Q\in \MM(\Delta)}b_{Q} \ux^Q\;\Big|\;b_{Q}\in \ZZ_q\llbracket T^{1/D}\rrbracket, \val_T(b_{Q})\to +\infty, \textrm{when}\ Q\to \infty\Big\}.$$
\end{definition}
Let $\psi_p$ denote the operator on $\mathbf{B}$ such that
$$\psi_p\Big(\sum\limits_{Q\in \MM(\Delta)}b_{Q} \ux^Q\Big): = \sum\limits_{Q\in \MM(\Delta)}b_{pQ} \ux^Q.$$

\begin{definition}
	Define
	\begin{equation}
	\label{E:psi}
	\psi := \sigma^{-1}_{\Frob}\circ\psi_p \circ E_{f}(\ux): \mathbf{B} \longrightarrow \mathbf{B}
	\end{equation}
		where $\sigma_\Frob$ represents the arithmetic Frobenius acting on the coefficients, and 
	$E_{f}(\ux)(g):=E_{f}(\ux)\cdot g$ for every $g\in \mathbf{B}$.
\end{definition} 

	Note that its $k$-th iterate satisfies
$$\psi^k=\sigma_{\Frob}^{-k}\circ\psi_p^k\circ \prod_{i=0}^{k-1}E_f^{\sigma_{\Frob}^i}(x_1^{p^i}, x_2^{p^i},\dots,x_n^{p^i}).$$

\begin{lemma}\label{N}
	Let $N$ be the matrix of $\psi$ acting on $\mathbf B$ with respect to the basis $ \{\ux^Q\}$, then the entries $$N_{ Q',Q}=E(\hat{a}_\calO\pi)e_{pQ'-Q}(T),$$
	where $e_{pQ'-Q}(T)$ is defined in \eqref{E:Ef(x)}.
\end{lemma}
\begin{proof}
From 
\begin{multline*}
	\psi_p \circ E_{f}(\ux)\big( \ux^Q\big) =\psi_p\Big( E(\hat{a}_\calO\pi)\sum_{Q''\in \MM(\Delta)} e_{Q''}(T) \ux^{Q''+Q}\Big)\\
= E(\hat{a}_\calO\pi)\sum_{\substack{Q''\in \MM(\Delta)\\Q+Q''=pQ'}} e_{Q''}(T)  \ux^{Q'}= \sum_{Q'\in \MM(\Delta)} E(\hat{a}_\calO\pi)e_{pQ'-Q}(T)  \ux^{Q'},
\end{multline*}
	 we complete the proof.
\end{proof}

%Explicitly, the matrix of $\psi$ with respect to the basis $\Gamma:=\{1,x,x^2,\dots\}$ is given by
%\begin{equation}
%	\label{E:explicit N}
%	N=\big( e_{mp-n}\big)_{m,n\geq 0} =\begin{pmatrix} 
%		e_0&0&\cdots&0& 0 &\cdots &0&\cdots\\
%		e_p & e_{p-1} & \cdots & e_0 &  0  & \cdots  & 0 & \cdots\\
%		e_{2p} & e_{2p-1} & \cdots & e_p & e_{p-1}  & \cdots & e_0 & \cdots\\
%		\vdots & \vdots & \ddots & \vdots & \vdots & \vdots & \ddots  & \ddots\\
%		e_{mp} & e_{mp-1} & \cdots & e_{mp-p} & e_{mp-p-1} & \cdots & e_{mp-2p} & \cdots\\
%		\vdots & \vdots & \ddots & \vdots & \vdots & \ddots & \vdots & \ddots
%	\end{pmatrix}.
%\end{equation}
%
%The operator $\tau^{-1}\circ \psi$ is $\tau^{-1}$-linear, but its $a$-th iteration $(\tau^{-1}\circ \psi)^a$ is linear since 
%$\tau^a$ acts trivially on $\Zp\llbracket \pi\rrbracket$. For the same reason, $\tau^a(N) =N$. 
Recall that $m = [\FF_q:\FF_p]$ and $C_f^*(T,s)=\sum\limits_{k=0}^\infty u_{k}( T) s^k \in \ZZ_p\llbracket T, s\rrbracket.$

\begin{theorem}[Dwork Trace Formula]
	
	For every integer $k\geq 1$, we have
	$$(-1)^{n-1}(q^k-1)^{n}S_f^*(k,\underline{\pi})=\Tr_{\mathbf{V}/\ZZ_q[\![ \pi]\!]}\big(\psi^{mk}\big).$$
\end{theorem}
\begin{proof} See \cite[Lemma~4.7]{liu-wan}.  
\end{proof}
We refer \cite{wan} for a thorough treatment of the universal Dwork trace formula.

%\begin{theorem}[Dwork Trace Formula]
%	For every integer $k>0$, we have
%	$$(q^k-1)^{-n}S_f^*(k,T)=\Tr_{\mathbf{B}/\ZZ_q[\![ \pi]\!]}\big(\psi^{mk}\big).$$
%\end{theorem}
%\begin{proof} This was proved by \cite[Lemma~4.7]{liu-wei}.  
%\end{proof}
%One can see \cite{wan} for a a thorough treatment of the universal Dwork trace formula. 

\begin{theorem}[Analytic trace formula]\label{determinant}
	The theorem above has an equivalent multiplicative form:
	\begin{equation}\label{dwork}
	\begin{split}
		C_f^*(T, s)=& \det\big(I-s \psi^m  \;|\; \mathbf{B}/\ZZ_q[\![ \pi]\!] \big).
	\end{split}
	\end{equation}
\end{theorem}
\begin{proof}
	It follows from Dwork trace formula. For proof, see \cite[Theorem~4.8]{liu-wan}.
	%	It follows from the following list of equalities 
	%	\begin{align*}
	%		C_f^*(\underline{\pi}, s) =& \exp \big( \sum_{k=1}^\infty \frac 1{1-p^k} S^*(k,\underline{\pi})\frac{s^k}{k} \big)  \\
	%		=&\exp \big( \sum_{k=1}^\infty {-\Tr_{\textbf{B}/\Zp[\![ \underline{\pi}]\!]}((\tau^{-1}\circ\psi)^{ak})}\frac{s^k}{k} \big) \\
	%		=& \det  \big(I-(\tau^{-1}\circ\psi)^{a}s\;\big|\;\mathbf{B}\big)\\
	%		=& \det\big(I-s \tau^{-1}(N) \tau^{-2}(N) \cdots \tau^{-a}(N)\big)
	%		\\
	%		=& \det \big(I-s \tau^{a-1}(N) \cdots \tau(N)N\big). \qedhere
	%	\end{align*}
	%\liang{The previous product expression is problematic because these operators do not commute with one another.}
\end{proof}

\begin{definition}
	The \emph{normalized Newton polygon} of $C_f^*(T, s)$, denoted by $\NP(f, T)_C$, is the lower convex hull of the set of points $\Big\{\left(i,\frac{\val_T(u_{i}(T))}{m}\right)\Big\}$.
\end{definition}

Now we recall the weight function and the usual Hodge polygon.
\begin{definition}\label{weight function}
	For an integer point $Q$ in $\ZZ^n$, assume that the line $\overline{OQ}$ intersects some surface of $\Delta$ at a point $Q'$. Then
	we define \emph{the weight} of $Q$ by the dilation $d\in \QQ$ such that $\overrightarrow{OQ}=d\cdot \overrightarrow{OQ'}$, and denote it by $w(Q)$.
	Note that $w(Q)$ is negative if $\overrightarrow{OQ}$ and $\overrightarrow{OQ'}$ have opposite directions.

	% 
	% For a multiset $S^\star$ of points in $\Cone(\Delta)$, we define 
	%	$m(S^\star)=$
\end{definition}

\begin{definition}\label{notation HP}
	Let $\WW_\ell$ denote a set consisting of $\ell$ elements of $\MM(\Delta)$ with minimal weights. The usual Hodge polygon, denoted by $\HP(\Delta)$, is the lower convex hull of 
	$$\Big\{\Big(\ell, \sum_{Q\in \WW_\ell}(p-1)w(Q)\Big)\Big\}.$$
\end{definition}

\begin{proposition}
	The Newton polygon $\NP(f, T)_C$ lies on or above $\HP(\Delta)$.
\end{proposition}
\begin{proof}
	See \cite[Theorem 1.3]{liu-wei}.
\end{proof}

\begin{definition}
	We call $p$ \emph{ordinary with respect to $\Delta$} if $p\equiv 1\pmod{D}.$
\end{definition}
%\begin{definition} Recall that we have defined $L_f^*(\chi,s)$ and $C^*_f(\chi, s)$ for each finite character $\chi:\Zp\to \CC_p^\times$ in the introduction. 
%	\begin{itemize}
%		\item [(a)] The \emph{normalized Newton polygon} of $L^*_f(\chi, s)^{-1}:=\sum \val_ix^i$, denoted by $\NP(f, \chi)_{L^{-1}}$, is the lower convex hull of the set of points $\left(i,\val_q(\val_i)p^{m_\chi-1}(p-1)\right)$.
%		\item [(b)] Similarly to $\NP(f, \chi)_{L^{-1}}$, we write $\NP(f, \chi)_C$ for the lower convex hull of the set of points $\left(i,\val_q(u_i)p^{m_\chi-1}(p-1)\right)$ and call it the \emph{normalized Newton polygon} of $C^*_f(\chi, s)$.
%		
%	\end{itemize}
%	
%\end{definition}	

%\begin{definition}\label{generic Newton polygon}
%	The \emph{generic Newton polygon} of a two dimensional convex polytope $\Delta$ is defined by
%	$$\GNP(\Delta):=\inf\limits_{\substack{\chi: \Zp/p^{m_\chi}\Zp\to \CC_p^\times\\\Delta=\Delta}}\{\NP(f, \chi)_{L^{-1}}\}.$$
%\end{definition}

%Liu and Wan gave a lower bound of $\NP(f, s)_C$ in \cite{liu-wan}, which is called its \emph{Hodge bound}. The main contribution of this paper is giving an improved Hodge bound and showing that it agrees with the $\GNP(\Delta)$ at infinite many points under some hypothesis.

\section{The improved Hodge polygon}\label{section 3} 
As we have explained in \S\ref{section 2} that for $f\in \mathcal F(\Delta)$, the Newton polygon $\NP(f,T)_C$ lies on or above the Hodge polygon introduced in Definition~\ref{notation HP}. However, unless $p$ is ordinary with respect to $\Delta$, this Hodge polygon is not expected to be sharp. In this section, we introduce an improved Hodge polygon which is again a lower bound of $\NP(f,T)_C$, and we prove that for a generic $f$, our improved Hodge polygon coincides with $\NP(f,T)_C$ at infinitely many points.

\begin{definition}\label{IHP}
	The \emph{improved Hodge polygon} of $\Delta$, denoted by $\IHP(\Delta)$, is the lower convex hull of the set of points 
	\begin{equation}\label{IHPP}
	\Big\{\Big(\ell,\sum_{Q\in \WW_\ell}(\big\lfloor w(pQ)\big\rfloor-\big\lfloor w(Q)\big\rfloor)\Big)\Big\},
	\end{equation}
	 where $\WW_\ell$ consists of $\ell$ elements of $\MM(\Delta)$ with minimal weights as in Definition~\ref{notation HP}.
\end{definition}

We will later give a simplified expression of this polygon at some particular points.

\begin{remark}\label{lamme property of w}
 Each point $Q \in \MM(\Delta)$ can be written as a rational linear combination $\sum\limits_{i=1}^n z_i \mathbf V_i$ of the basis vectors. It is straightforward to see that \begin{equation}\label{equation w1}
 w(Q) =\max\limits_{1\leq i\leq n}\{z_i\}.
 \end{equation} In particular, since each $z_i \in \frac 1D \ZZ$, we have $w(Q) \in \frac 1D \ZZ$  for every $Q\in \MM(\Delta)$.
 
		Moreover, the weight function is subadditive, namely, for every two points $Q_1, Q_2\in \MM(\Delta)$, we have
		\begin{equation}\label{inequality for w}
		w(Q_1)+w(Q_2)\geq w(Q_1+Q_2).
		\end{equation}
\end{remark}

\begin{proposition}\label{IHPHP}\hfill
\begin{enumerate}
	\item 	The improved Hodge polygon $\IHP(\Delta)$ lies on or above $\HP(\Delta)$ at $x=\mathbbm x_k^\pm$ for every $k\geq 1$.
	\item If there is a point $P_0=\sum\limits_{i=1}^n r_i\mathbf V_i\in \Delta^-$ and $j_1,j_2\in \{1,2,\dots,n\}$ such that \begin{enumerate}
		\item $r_{j_1}<r_{j_2}$, and
		\item $pr_{j_1}-\lfloor pr_{j_1}\rfloor>pr_{j_2}-\lfloor pr_{j_2}\rfloor$,
	\end{enumerate} then 
	 $\IHP(\Delta)$ lies strictly above $\HP(\Delta)$ at $x=\mathbbm x_k^\pm$ for every $k\geq 2$.
	 	\item 	If there is a point $P_0=\sum\limits_{i=1}^n r_i\mathbf V_i\in \Delta^-$ such that 
	 	$$\Big\{j\;\Big|\; r_j=\max\limits_{1\leq i\leq n}(r_i)\Big\}\cap \Big\{j\;\Big|\; pr_j-\lfloor pr_j\rfloor=\max\limits_{1\leq i\leq n}(pr_i-\lfloor pr_i\rfloor)\Big\}=\emptyset,$$
	 	then 
	 	$\IHP(\Delta)$ lies strictly above $\HP(\Delta)$ at $x=\mathbbm x_k^\pm$ for every $k\geq 1$.
\end{enumerate}
\end{proposition}
\begin{notation}\label{eta}
	For a point $Q$ in $\MM(\Delta)$ we write $Q\%$ for its residue in $\Delta^-$ modulo $\Lambda_\Delta$, and  set  
	\begin{eqnarray*}
	\eta:&\Delta^-\to \Delta^-\\
	&Q\mapsto (pQ)\% .
\end{eqnarray*}
be a permutation of $\Delta^-$.
\end{notation}

\begin{proof}[Proof of Proposition~\ref{IHPHP}]
	(1) Since any point $Q\in \Delta_k^+\backslash\Delta_k^-$ has integer weight, we get $$\big\lfloor w(pQ)\big\rfloor-\big\lfloor w(Q)\big\rfloor= w(pQ)- w(Q).$$ Therefore, we only need to prove this proposition for $x=\mathbbm x_k^-$.
	
	Since $\Delta_k^-$ can be decomposed into a disjoint union of shifts of $\Delta^-$ by points in $\Lambda_{\Delta}$, we reduce the question to show 
	 \begin{equation}\label{eq a}
	 \sum_{Q-Q_1\in \Delta^-}\Big(\big\lfloor w(pQ)\big\rfloor-\big\lfloor w(Q)\big\rfloor- w(pQ)+w(Q)\Big)\geq 0 \textrm{\ for every\ }
	 Q_1=\sum\limits_{i=1}^n m_i\mathbf V_i\in \Lambda_{\Delta}.
	 \end{equation}

		Let $Q-Q_1=\sum\limits_{i=1}^n r_i\mathbf V_i\in \Delta^-$.
We set $$m_\max:=\max_{1\leq j\leq n}(m_j)\quad\textrm{and} \quad S=:\{1\leq i\leq n\;|\; m_i=m_\max\}. $$
Take $j\in S$ such that $r_j=\max\limits_{i\in S}(r_i)$.
Then we have 
\begin{equation}\label{aa}
\big\lfloor w(Q)\big\rfloor+w(pQ)=m_j +pr_j+pm_j=w\Big(pQ+Q_1\Big).
\end{equation}	
	 Since $\eta$ is a permutation of $\Delta^-$, we know that 
	 \begin{equation}\label{ab}
	 \sum_{Q-Q_1\in \Delta^-} w(Q)=\sum_{Q-Q_1\in \Delta^-} w(Q_1+\eta(Q-Q_1)).
	 \end{equation}

Since $pQ\equiv \eta(Q-Q_1) \pmod{\Lambda_{\Delta}}$, we know 
\begin{equation}\label{ac}
w\Big(pQ-\eta(Q-Q_1)\Big)=(pr_j+pm_j)-(pr_j-\lfloor pr_j\rfloor)
=pm_j+\lfloor pr_j\rfloor=
\big\lfloor w(pQ)\big\rfloor.
\end{equation}
Combining \eqref{inequality for w}, \eqref{aa}, \eqref{ab} and \eqref{ac}, we get
	\begin{align*}
	&\sum_{Q-Q_1\in \Delta^-}\Big(\big\lfloor w(pQ)\big\rfloor-\big\lfloor w(Q)\big\rfloor- w(pQ)+w(Q)\Big)\\
		\stackrel{\eqref{ab}}{=}	&\sum_{Q-Q_1\in \Delta^-}\Big( \big\lfloor w(pQ)\big\rfloor-\big\lfloor w(Q)\big\rfloor- w(pQ)+w\Big(Q_1+\eta(Q-Q_1)\Big)\Big)\\
		\stackrel{\eqref{ac}}{=}	&\sum_{Q-Q_1\in \Delta^-}\Big( w\Big(pQ-\eta(Q-Q_1)\Big)-\big\lfloor w(Q)\big\rfloor- w(pQ)+w\Big(Q_1+\eta(Q-Q_1)\Big)\Big)\\
\stackrel{\eqref{aa}}{=}&\sum_{Q-Q_1\in \Delta^-}	\Big(w\Big(pQ-\eta(Q-Q_1)\Big)+w\Big(Q_1+\eta(Q-Q_1)\Big)-w\Big(pQ+Q_1\Big)\Big)\\
	\stackrel{\eqref{inequality for w}}{\geq}&0.
	\end{align*}

(2)  We put $Q_1=\mathbf V_{j_1}+\mathbf V_{j_2}$ and $Q=Q_1+P_0$.
From the assumptions (a) and (b), we have 
\begin{multline*}
w\Big(pQ-\eta(Q-Q_1)\Big)+w\Big(Q_1+\eta(Q-Q_1)\Big)-w\Big(pQ+Q_1\Big)\\=p+\lfloor pr_{j_2}\rfloor +(1+pr_{j_1}-\lfloor pr_{j_1}\rfloor)-(p+pr_2+1)
\\=pr_{j_1}-\lfloor pr_{j_1}\rfloor-(pr_{j_2}-\lfloor pr_{j_2}\rfloor)>0,
\end{multline*}
which completes the proof.

(3) The proof of (3) is similar to (2).
\end{proof}

\begin{example}
	Let $p=29$, $\mathbf V_1=(3,0)$ and $\mathbf V_2=(0,2)$. Then we have	$$\WW_6=\Delta_1^-=\{(0, 0),(0, 1),(1, 0),(1,1),(0, 2), (1,2) \},$$
	and the chart of weight of points in $\Delta^-$:
	\begin{center}
		\begin{tabular}{ | c|c|c|c|c|c|c|} 
			\hline
			$P_0$& $(0, 0)$&$(0, 1)$& $(1, 0)$&$(1,1)$&$(0,2)$&$(1, 2)$ \\ 
			\hline 
			$w(P_0)$	&$0$&$\frac{1}{3}$& $\frac{1}{2}$&$\frac{1}{2}$&$\frac{2}{3}$&$\frac{2}{3}$ \\ [1ex]
			\hline
			$w(pP_0)$& $0$&$\frac{29}{3}$& $\frac{29}{2}$&$\frac{29}{2}$&$\frac{58}{3}$&$\frac{58}{3}$ \\ [1ex]
			\hline
		\end{tabular}.
	\end{center}
Computing the left hand side of \eqref{eq a} for $Q_1=(0,0)$, we have
\[\sum_{P_0\in \Delta^-}\Big(\big\lfloor w(pP_0)\big\rfloor-\big\lfloor w(P_0)\big\rfloor- w(pP_0)+w(P_0)\Big)=\frac{1}{3}>0.
\]
\end{example}

We next give an estimate of the $T$-adic valuation of each entry of the matrix $N$. For this, we need to control the $T$-adic valuation of each $e_Q(T)$.
\begin{lemma}\label{L:estimate of Ef(x)}
	\hfill
	\begin{itemize}
		\item[(1)] Recall that $\prod\limits_{P\in \Delta^+\backslash \{\calO\}} E(\hat a_{P} \pi  \ux^P)$ expands as $\sum\limits_{Q\in \MM(\Delta)} e_{Q}(T) \ux^Q$. Then we have $$e_{O}(T)=1\quad \textrm{and}\quad \val_T(e_{Q}(T))\geq \big\lceil w(Q)\big\rceil\ \textrm{for all}\ Q \in \MM(\Delta).$$
		\item[(2)] Recall that $	\prod\limits_{P\in \Delta^+\backslash \{\calO\}} E(\widetilde a_{P} \pi  \ux^P)$ expands as $\sum\limits_{Q\in \MM(\Delta)} \widetilde{e}_{Q}(T) \ux^Q$. Then we have $$\widetilde e_{O}(T)=1\quad \textrm{and}\quad \val_T(\widetilde e_{Q}(T))\geq \big\lceil w(Q)\big\rceil\ \textrm{for all}\ Q \in \MM(\Delta).$$
	\end{itemize}
\end{lemma}

\begin{proof}
	(1) 
	We will only prove (1) since the proof of (2) is similar.
	
	It follows from Definition~\ref{E:Ef(x)} that $e_{O}(T)=1$.
	
	For each $P\in \Delta^+\backslash \{\calO\}=\Delta\cap \ZZ^n\backslash \{\calO\}$, we expand $E(\hat{a}_{P} \pi  \ux^{P})$ to a power series in variables $x_1, x_2,\dots,x_n$, and get $$\prod_{P\in \Delta^+\backslash \{\calO\}}E(\hat{a}_{P} \pi  \ux^{P})=\sum_{\vec{j}\in \ZZ_{\geq 0}^{\Delta^+\backslash \{\calO\}}}\Big(\prod_{P\in \Delta^+\backslash \{\calO\}}c_{j_P} (\hat{a}_{P} \pi  \ux^{P})^{j_P}\Big), $$
	where $\{\hat{a}_{P}\}$ is the set of coefficients of $\hat f(\ux)$ and $c_i\in \ZZ_p$ is the $\pi^i$ coefficient of $E(\pi)$.

Expanding this product and the sum, we deduce
\begin{multline*}
	e_{Q}(T)=\sum_{\Big\{\vec{j}\in \ZZ_{\geq 0}^{\Delta^+\backslash \{\calO\}}\;\Big|\;\sum\limits_{P\in \Delta^+\backslash \{\calO\}}j_PP=Q\Big\}}\Big(\prod_{P\in \Delta^+\backslash \{\calO\}}\Big(c_{j_P} (\hat{a}_{P} \pi)^{j_P}\Big)\Big)\\
	=\sum_{\Big\{\vec{j}\in \ZZ_{\geq 0}^{\Delta^+\backslash \{\calO\}}\;\Big|\;\sum\limits_{P\in \Delta^+\backslash \{\calO\}}j_PP=Q\Big\}}\Big(\prod_{P\in \Delta^+\backslash \{\calO\}}\Big(c_{j_P}(\hat{a}_{P})^{j_P} \Big) \pi^{\sum\limits_{P\in \Delta^+\backslash \{\calO\}}j_P}\Big).
\end{multline*}

	Since each point in  $\Delta^+\backslash \{\calO\}$ has weight less or equal to $1$, for each $\vec{j}$ such that $\sum\limits_{P\in \Delta^+\backslash \{\calO\}}j_PP=Q$, we have \[
	\val_T\Big(\prod_{P\in \Delta^+\backslash \{\calO\}}c_{j_P}(\hat{a}_{P})^{j_P}  \pi^{\sum\limits_{P\in \Delta^+\backslash \{\calO\}}j_P}\Big)\\=\sum\limits_{P\in \Delta^+\backslash \{\calO\}}j_P\geq \sum\limits_{P\in \Delta^+\backslash \{\calO\}}j_Pw(P)\geq w(Q).\]
	Note that $\val_T(\pi)=1$.
	Therefore, we immediately obtain $\val_T(e_{Q}(T))\geq w(Q)$. Since $\val_T(e_{Q}(T))$ is an integer, we have 
	\[\val_T(e_{Q}(T))\geq \big\lceil w(Q)\big\rceil.\qedhere\]
\end{proof}

\begin{corollary}\label{NQQ'}
	If both $Q$ and $Q'$ belong to $\MM(\Delta)$, then \begin{equation}\label{NQQ}
	\val_T(N_{Q',Q})=\val_T(e_{pQ'-Q})\geq \lfloor w(pQ')\rfloor-\lfloor w(Q)\rfloor,
	\end{equation}
	where $N_{Q',Q}$ is the entry of matrix $N$ as in Lemma~\ref{N}.
\end{corollary}
\begin{proof}
	Since $\val_T(E(\hat a_\calO))=0$, we have 
	$$\val_T(N_{Q',Q})=\val_T(E(\hat a_\calO)\cdot e_{pQ'-Q})=\val_T(e_{pQ'-Q}).$$
	We assume that $pQ'-Q\in \MM(\Delta)$, for otherwise $\val(e_{pQ'-Q})=\infty$, which leads to \eqref{NQQ} directly.

	By Lemma~\ref{L:estimate of Ef(x)}, we have
\begin{multline*}
\val(e_{pQ'-Q})\geq \lceil w(pQ')-w(Q)\big\rceil\\
=\Big\lceil w(pQ')-w(Q)-\Big(\lfloor w(pQ')\rfloor -\lfloor w(Q)\rfloor\Big) \Big\rceil+\lfloor w(pQ')\rfloor -\lfloor w(Q)\rfloor\\
\geq\lfloor w(pQ')\rfloor -\lfloor w(Q)\rfloor.\qedhere
\end{multline*}
\end{proof}

\begin{notation}\label{det}
	\hfill
	\begin{enumerate}
		\item Let $\SS_1$ and $\SS_2$ be two sets of the same cardinality $\ell$. We denote by $\Iso(\SS_1,\SS_2)$ the set of bijections from $\SS_1$ to $\SS_2$.
		If the elements in these two sets are labeled as $$\SS_1:=\{Q_1,Q_2,\dots,Q_\ell\} \quad\textrm{and}\quad \SS_2:=\{Q_1',Q_2',\dots,Q_\ell'\},$$ for a bijection $\tau\in \Iso(\SS_1, \SS_2)$ such that  $\tau(Q_i)=Q'_{j_i}$ we put $$\sgn(\tau)=\sgn(j_1,j_2,\dots,j_\ell).$$

		\item 	For a function $G$ on $\SS_1\times\SS_2$, we put $$\det\Big(G(Q,Q')\Big)_{Q,Q'\in \SS_1\times\SS_2}:=\pm\sum_{\tau\in\Iso(\SS_1, \SS_2)}\sgn(\tau)\prod\limits_{Q\in \SS_1} G{(Q,\tau(Q))},$$
	\end{enumerate}
which is independent to the order of elements in $\SS_1$ and $\SS_2$ up to a sign.
\end{notation}

\begin{notation}\hfill
\begin{enumerate}
	\item 	For the rest of this section, we shall consider multisets, i.e. sets of possibly repeating elements. They are marked with a superscript star to be distinguished from usual sets. 
	\item 	The disjoint union of two multisets $\SS^\star$ and ${\SS'}^\star$ is denoted by $\SS^\star\uplus {\SS'}^\star$ as a multiset. 
	
	\item For a set $\SS$, we write ${\SS^\star}^m$ (resp. ${\SS^\star}^{\infty}$) for the union of $m$ (resp. countably infinite) copies of $\SS$ as a multiset. 
\end{enumerate}

\end{notation}

\begin{notation}\label{Iso and Prem}
	For two multisets $\SS_1^\star$ and $\SS_2^\star$ of the same cardinality, we denote by $\Iso(\SS_1^\star, \SS_2^\star)$ the set of bijections (as multisets) from $\SS_1^\star$ to $\SS_2^\star$. When $\SS_1^\star=\SS_2^\star=\SS^\star$, we simply set $\Iso(\SS^\star):=\Iso(\SS^\star, \SS^\star)$.
\end{notation}

\begin{definition}\label{definition of h}
	Let $\SS^\star$ be a subset of $\MM(\Delta)^{\star\infty}$. We define
	\begin{equation}\label{defintion of h SS}
	h(\SS^\star):=\sum\limits_{Q\in \SS^\star} \big\lfloor w(pQ)\big\rfloor-\big\lfloor w(Q)\big\rfloor. 
	\end{equation}
	The expression on the right hand side will be related to the estimate in Corollary~\ref{NQQ'}, and to the expression \eqref{IHPP}.
	
		If $\SS^\star$ belongs to $\MM(\Delta)$, we suppress the star from the notation.
	
\end{definition}

%\begin{lemma}\label{T-adic bound of det}
%	We have $$\val_T(\det(\SS))\geq h(\SS)\quad\textrm{and}\quad \val_T(u_\ell(T))\geq \min\limits_{\SS\subset \mathscr M_\ell}h(\SS).$$
%\end{lemma}
%\begin{proof}
%	It follows directly from Lemma~\ref{E:Ef(x)} and the definition of $h$ in	\eqref{defintion of h SS} and \eqref{h function}.
%\end{proof}

%
%\begin{definition}
%	 (1) The normalized Newton polygon of $C_f'(T)$ is the lower convex hull of the set of points $\left(i,\frac{\val_T(u_i)}{a}\right)$;
%	 
%	\noindent(2) We denote $\LP(\Delta)$ be a lower convex hull of the set of points $$\Big\{(i, \frac{1}{a}\min\limits_{\SS\subset \mathscr M_i}h(\SS))\Big\}.$$
%\end{definition}
%\begin{lemma}
%	The normalized Newton polygon of $C_f'$ lies above $\LP(\Delta)$.
%\end{lemma}
%\begin{proof}
%	It follows from Lemma~\ref{T-adic bound of det}. 
%\end{proof}
%
%By multiplying the height of $\LP(\Delta)$ at each point by $a$, we get a polygon, which is called  the \emph{improved Hodge polygon} of $\Delta$ and denoted by $\IHP(\Delta)$. By Proposition~\ref{determinant}, it is easy to check the following proposition.

\begin{notation}
	We denote $\mathscr M_\ell(k)$ to be the set consisting of all sub-multisets of ${\MM(\Delta)^\star}^k$ of cardinality $k\ell$. For simplicity, we put $\mathscr M_\ell:=\mathscr M_\ell(1)$.
\end{notation}
\begin{remark}
	It is clear that $\IHP(\Delta)$ is the lower convex hull of the set of points $\Big\{\left(\ell,\min\limits_{\SS\in \mathscr M_\ell}h(\SS)\right)\Big\}$.
\end{remark}

Now we gather more information of $\IHP(\Delta)$.
Recall that $D$ is the smallest positive integer that satisfies \eqref{equation w}.
	\begin{lemma}\label{first lemma}
		Let $\SS_1, \dots, \SS_m \in \mathscr M_{\ell}$ so that their disjoint union ${\mathbb{S}}^\star =\biguplus\limits_{j=0}^{m-1} \SS_{j}\in \mathscr M_\ell(m)$. 
		Then the following two statements are equivalent.
		\begin{itemize}
			\item[(1)] 	The minimum of $h({\mathbb{S}'}^\star)$ over all ${\mathbb{S}'}^\star \in \mathscr M_\ell(m)$ is achieved by $\mathbb{S}^\star$.
			\item[(2)]
			For every $1\leq i\leq m$, the sum of weights $\sum\limits_{Q \in \mathbb{S}_i} w(Q)$ achieves the minimum of $\sum\limits_{Q \in \mathbb{S}'} w(Q)$ over all $\mathbb{S}' \in \mathscr M_\ell.$
		\end{itemize}
\end{lemma}
\begin{proof}
	It is obvious that $h$ is additive. Hence, without loss of generality, we assume $m=1$ in the proof.
	
	First, we claim that for $Q$ and $Q'$ two points in $\MM(\Delta)$, if 
	$w(Q)>w(Q')$, then \begin{equation}\label{eqQQ'}
	\lfloor w(pQ)\rfloor-\lfloor w(Q)\rfloor-(\lfloor w(pQ')\rfloor-\lfloor w(Q')\rfloor)>0.
	\end{equation}
	Indeed, by Definition~\ref{definition of D}, if $w(Q)>w(Q')$, then $$w(Q)\geq w(Q')+\frac{1}{D}.$$
	Since we assume $p>D(n+4)$ in Hypothesis~\ref{hypothesis}, we have 
\begin{multline*}
	\lfloor w(pQ)\rfloor-\lfloor w(Q)\rfloor-(\lfloor w(pQ')\rfloor-\lfloor w(Q')\rfloor)\\
	> w(pQ)-w(pQ')- w(Q)+w(Q')-2
	\geq \frac{p-1}{D}-2
	>0.
\end{multline*}
Therefore, if $\SS' \in\mathscr M_\ell$ such that $h(\SS')$ is minimal over all $\SS\in \mathscr M_\ell$, then $\sum\limits_{Q \in \SS'}w(Q)$ takes the minimum of $\sum\limits_{Q \in \SS} w(Q)$ for all $\SS\in \mathscr M_\ell.$

Let $\SS'$ and $\SS''$ be  two subsets of $\mathscr M_\ell$ such that $\sum\limits_{Q \in \mathbb{S}'} w(Q)$ and $\sum\limits_{Q \in \mathbb{S}''} w(Q)$ reach the minimum of $\sum\limits_{Q \in \mathbb{S}} w(Q)$ over all $\SS\in \mathscr M_\ell$. Clearly, the two multisets $\{w(Q)\;|\;Q\in \SS'\}$ and $\{w(Q)\;|\;Q\in \SS''\}$ are the same, which implies $h(\SS')=h(\SS'')$ and this lemma.
\end{proof}

\begin{proposition}\label{corollary smallest set}
For every integer $\ell\geq 1$, let $\WW_\ell$ be a set of $\ell$ elements in $ \MM(\Delta)$ with minimal weights. Then we have
% Then $\min\limits_{\mathbb{S} \in \mathscr M_\ell} h(\mathbb{S})= h(\WW_\ell).$
\begin{enumerate}
	\item	
	$\min\limits_{\SS^\star\in \mathscr M_\ell(m)}h(\SS^\star)=mh(\WW_\ell).$
	\item Order the elements in $\MM(\Delta)$ by their weights (in the non-decreasing order): $P_1, P_2,\dots.$ Suppose that $n_1, n_2, \dots$ are exactly the indices such that $w(P_{n_i+1}) > w(P_{n_i})$. Then $\IHP(\Delta)$ has vertices $(n_i, h(\WW_{n_i}))$ for every $i\geq 1$.
\end{enumerate}	
\end{proposition}
	\begin{proof}
	(1) From the proof of Lemma~\ref{first lemma}, we know that the minimum on the left is achieved when ${\mathbb{S}}^\star =\biguplus\limits_{j=0}^{m-1} \SS_{j}$ with $\SS_j=\WW_\ell$. Since
	$h$ is additive, we have
	$$h\Big(\biguplus\limits_{j=0}^{m-1}\SS_j\Big)=\sum_{j=0}^{m-1}h(\SS_j)=mh(\WW_\ell).$$
	
	(2) Since $\{w(P_n)\}$ is non-decreasing, by Lemma~\ref{first lemma}, the improved Hodge polygon $\IHP(\Delta)$ passes through the point $\left(\ell, \sum\limits_{i=1}^\ell (\lfloor w(pP_i)\rfloor-\lfloor w(P_i)\rfloor)\right)$ for every $\ell\geq 1$. Since $w(P_{n_i+1}) > w(P_{n_i})$, by \eqref{eqQQ'}, we have 	$$\lfloor w(pP_{n_i+1})\rfloor-\lfloor w(P_{n_i+1})\rfloor-\left(\lfloor w(pP_{n_i})\rfloor-\lfloor w(P_{n_i})\rfloor\right)>0.$$
	Therefore, $(n_i, h(\WW_{n_i}))$ is a vertex of $\IHP(\Delta)$ for every $i\geq 1$.
\end{proof}
\begin{corollary}
			\hfill
	\begin{enumerate}
		\item	The height of the improved Hodge polygon at $x=\mathbbm{x}_k^\pm$ satisfies 
	$$\min\limits_{\SS\in \mathscr M_{\mathbbm x_k^\pm}} h(\SS)=h(\Delta^\pm_k).$$
	\item Every Newton slope of $\IHP(\Delta)$ before point $x=\mathbbm x^-_k$ is strictly less than $k(p-1)$.
	\item Every slope after the point $x=\mathbbm x^+_k$ is strictly greater than $k(p-1)$.
	\item Every Newton slope of $\IHP(\Delta)$ between points $x=\mathbbm x^-_k$ and $x=\mathbbm x^+_k$ is equal to $k(p-1)$.
	\item For every $k\geq 1$, the point $(\mathbbm x_k^\pm, h(\Delta_k^\pm))$ is a vertex of $\IHP(\Delta)$. 
\end{enumerate}
\end{corollary}
\begin{proof}
	They are straightforward corollaries of  Proposition~\ref{corollary smallest set}.
%			(1) It follows directly from Proposition~\ref{corollary smallest set} (2).
%			
%			(2) Let $\ell< \mathbbm x_k^-$. 
%			By Lemma~\ref{first lemma}, the subset $\SS_0 \in \mathscr M_\ell$ that achieves the minimum in $\min\limits_{\substack{\#\SS =\mathscr M_\ell}} h(\SS)$ is contained in $\Delta_k^-$.
%		
%			Therefore, we have 
%			\begin{equation}\label{equation 3}
%			\begin{split}
%			\frac{1}{\mathbbm x_k^--\ell}\Big(h(\Delta_k^-)-h(\SS_0)\Big)=&\frac{1}{\mathbbm x_k^--\ell}h(\Delta_k^--\SS_0)\\
%			\leq& \frac{1}{\mathbbm x_k^--\ell}\Big(\sum\limits_{Q\in \Delta_k^--\SS_0}\lfloor w(pQ)\rfloor-\lfloor w(Q)\rfloor \Big)\\
%			\leq&\frac{1}{\mathbbm x_k^--\ell}\Big( \sum\limits_{Q\in \Delta_k^--\SS_0} w(pQ)- w(Q)+1 \Big)\\
%			\leq&  (p-1)(k-\frac{1}{D})+1\\
%			<& k(p-1),
%			\end{split}
%			\end{equation}
%		which implies (2).
%		
%	
%	(3) Similar to (2), for $\ell> \mathbbm x_k^+$, we have $$\frac{1}{\ell-\mathbbm x_k^+}\Big(\min\limits_{\SS\in \mathscr M_\ell} h(\SS)-h(\Delta_k^+)\Big)>k(p-1),$$
%	which completes the proof.
%	
%		(4) Similar to (2) and (3), for $\mathbbm x_k^-<\ell< \mathbbm x_k^+$, we have $$\frac{1}{\ell-\mathbbm x_k^-}\Big(\min\limits_{\SS\in \mathscr M_\ell} h(\SS)-h(\Delta_k^-)\Big)=\frac{1}{\mathbbm x_k^+-\ell}\Big(\min\limits_{\SS\in \mathscr M_\ell} h(\Delta_k^+)-h(\SS)\Big)= k(p-1).$$ 
%		Combining it with (2) and (3), we shows (4).
%		
%	(5) It directly follows from (2)-(4).
	\end{proof}

\begin{notation}\label{notation 3.14}
	For a subset $I\subseteq\{1,2,\dots,n\}$ we put
	$$\Delta^-(I)=\Big\{\sum\limits_{i=1}^n r_i\mathbf V_i\in \Delta^-\Big|\;
r_i=0 \textrm{\ if\ } i\in I;\ 
r_i\in (0,1)	 \textrm{\ otherwise}
	\Big\}\subseteq \Delta^-.$$

\end{notation}
%\begin{remark}
%	Note that it is compatible to the definition of $\Delta_k^\pm(I)$ in Notation~\ref{Delta(I)}.
%\end{remark}

\begin{lemma}\label{polynomialforx}
	For every integer $k\geq 1$ we have
	\begin{enumerate}
		\item $\mathbbm x^-_k=k^n\Vol(\Delta),$ and 
		\item $\mathbbm x^+_k=\sum\limits_{I\subseteq \{1,2,\dots,n\}}\#\Delta^-(I)k^{\#I}{(k+1)}^{n-\#I}.$
	\end{enumerate}
	
\end{lemma}

\begin{proof}
	(1) 	Every point in $Q\in \Delta_k^-$ can be uniquely written as a sum of a point  $P_0\in \Delta^-$ and a point  $Q_1=\sum\limits_{i=1}^n m_i\mathbf V_i\in \Lambda_{\Delta}$  for an $n$-dimensional vector $\underline m\in \{0,\dots,k-1\}^n$.
	It implies $\mathbbm x^-_k=\#(\Delta^-) k^n.$
	Since $\Delta$ is a parallelotope, we have $\#\Delta^-=\Vol(\Delta)$, and hence
	$\mathbbm x^-_k=k^n\Vol(\Delta).$
	
	(2) Let $I\subset \{1,2,\dots,n\}$, $P_0\in \Delta^-(I)$ and $Q\in \MM(\Delta)$ such that $Q\equiv P_0\pmod{\Lambda_\Delta}$. Write $Q=P_0+Q_1$ for some $Q_1=\sum\limits_{i=1}^n m_i\mathbf V_i\in \Lambda_{\Delta}$. Then 
	$Q$ belongs to  $\Delta_k^+$ if and only if $0\leq m_i\leq k$ for every $i\in I$ and $0\leq m_i\leq k-1$ for every $i\notin I$. Therefore, there are $k^{\#I}(k+1)^{n-\#I}$ points in $\Delta_k^+$ with the residue $P_0$ module $\Lambda_{\Delta}$. It completes the proof.
\end{proof}

%\begin{notation}
%	for every non-empty set $I\subset\{1,2,\dots,n\}$ and $k\geq 1$, we put 
%	$$U(k;I)=\Big\{\sum_{i=1}^n m_i\mathbf V_i\in \Lambda_{\Delta}\;\Bigg|\;
%	 \begin{aligned}
%	&	m_i=k\quad  \textrm{if}\ i\in I,\\
%	&	m_i< k\quad  \textrm{otherwise}
%	\end{aligned}
%	 \Big\}.$$
%\end{notation}

%\begin{lemma}\label{lemma h(S1)-h(S2)}
%	Let $P_0=\sum\limits_{i=1}^{n}r_i\mathbf{V}_i\in \Delta^-$, and let $I$ be a non-empty set $\{1,2,\dots,n\}$. Let $k_1$ and $k_2$ be two positive integers.  Then for every $Q_1\in P_0+U(k_1;I)$ and $P_2\in P_0+ U(k_2;I)$, we have 
%	\begin{equation}
%	h(Q_1)=h(P_2)+(k_1-k_2)(p-1),
%	\end{equation}
%	where $h(P)=\lfloor w(pP)\rfloor-\lfloor w(P)\rfloor.$
%\end{lemma}
%
%\begin{proof}
%	
%	We assume $r_{s}=\max\{r_i\;\big|\;i\in I\}$
%	for some $s\in I$, and put
% $$Q_1=P_0+\sum_{i=1}^n m_{1,i}\mathbf V_i\quad \textrm{and}\quad  
%	P_2=P_0+\sum_{i=1}^n m_{2,i}\mathbf V_i.$$
%	
%
%By Remark~\ref{lamme property of w} (1), we have
%$$\lfloor w(p(P_0+\sum_{i=1}^n m_{1,i}V_i))\rfloor-\lfloor w(P_0+\sum_{i=1}^n m_{1,i}V_i)\rfloor=\lfloor pr_s\rfloor +pk_1-k_1$$
%and 
%$$\lfloor w(p(P_0+\sum_{i=1}^n m_{2,i}V_i))\rfloor-\lfloor w(P_0+\sum_{i=1}^n m_{2,i}V_i)\rfloor=\lfloor pr_s\rfloor +pk_2-k_2,$$
%which implies 
%\begin{align*}
%h(Q_1)-h(P_2)=&pk_1-k_1-pk_2+k_2\\
%=&(k_1-k_2)(p-1).\qedhere
%\end{align*}
%\end{proof}
% 

% \begin{proposition}\label{proposition G}
% 
% \end{proposition}
% \begin{proof}
% 	See Faulhaber's formula.
% \end{proof}

\begin{lemma}\label{lemma polynomial for h}
	The functions $h(\Delta_k^-)$ and $h(\Delta_k^+)$ are both  polynomials in $k$ of degree $n+1$, i.e. for every $k \geq 1$ we have $$h(\Delta_k^\pm)=\sum_{i=0}^{n+1} A^\pm_ik^i,$$
	where $A^\pm_i$  are integers which depend only on $\Delta$.
\end{lemma}

\begin{proof}
	By Lemma~\ref{polynomialforx}, we have
\begin{multline*}
	h(\Delta_k^+)=h(\Delta_k^-)+(p-1)k(\mathbbm x_k^+-\mathbbm x_k^-)\\
	=h(\Delta_k^-)+(p-1)k\sum_{I\subseteq \{1,2,\dots,n\}}\left(\#\Delta^-(I)k^{\#I}{(k+1)}^{n-\#I}-\#\Delta^-(I)k^n\right).
\end{multline*}
	Hence, 
	it is enough to prove the statement for $h(\Delta_k^-)$.
	For every nonempty subset $I\subseteq\{1,2,\dots,n\}$ and $k\geq 1$, we put 
	$$U(k;I)=\Bigg\{\sum_{i=1}^n m_i\mathbf V_i\in \Lambda_{\Delta}\;\Bigg|\;
	\begin{aligned}
	&	m_i=k\quad  \textrm{if}\ i\in I\\
	&	m_i< k\quad  \textrm{otherwise}
	\end{aligned}
	\Bigg\}.$$
	Clearly, a point $Q\in \MM(\Delta)$ belongs to $\Delta_k^-\backslash \Delta^-$ if and only if there exist $P_0\in \Delta^-$, $\ell\leq k-1$, and a nonempty subset $I\subseteq \{1,2,\dots,n\}$ such that $Q\in P_0+U(\ell;I)$. 
	Therefore, we have the following decomposition:
	$$\Delta_k^-=\Delta^-\sqcup\coprod_{P_0\in \Delta^-}\coprod_{\substack{I\subseteq\{1,2,\dots,n\}\\I\neq \emptyset}}\coprod_{\ell=1}^{k-1}\left( P_0+U(\ell;I)\right).$$
	
Since $h$ is additive, we obtain
		$$h(\Delta_k^-)=h(\Delta^{-})+\sum_{P_0\in \Delta^-}\sum_{\substack{I\subseteq\{1,2,\dots,n\}\\I\neq \emptyset}}\sum_{\ell=1}^{k-1} h(P_0+U(\ell;I)).$$
		
	It is enough to show that for a fixed $P_0$ and a nonempty subset $I\subseteq I(P_0)$, the sum $\sum\limits_{\ell=1}^{k-1}h(P_0+U(\ell;I))$ is a polynomial in $k$ of degree $\leq n+1$.

	Now we prove that every pair of points $P_1=P_0+\sum\limits_{i=1}^n m_{1,i}\mathbf V_i\in P_0+U(k_1;I)$ and $P_2=P_0+\sum\limits_{i=1}^n m_{2,i}\mathbf V_i\in P_0+ U(k_2;I)$ satisfies
	\begin{equation}\label{equationP1P2}	h(P_1)=h(P_2)+(k_1-k_2)(p-1),
	\end{equation}
	where $h(P):=\lfloor w(pP)\rfloor-\lfloor w(P)\rfloor.$
	
	Let $P_0=\sum\limits_{i=1}^{n}r_i\mathbf{V}_i$ and $r_{s}=\max\limits_{i\in I}\{r_i\}$
	for some $s\in I$.
	By \eqref{equation w1}, we have
	$$\Big\lfloor w(p(P_0+\sum_{i=1}^n m_{j,i}V_i))\Big\rfloor-\Big\lfloor w(P_0+\sum_{i=1}^n m_{j,i}V_i)\Big\rfloor=\lfloor pr_s\rfloor +pk_j-k_j$$
for $j=1,2$,
	which implies \eqref{equationP1P2}.

	 We choose a representative $P'\in U(1;I)$ and put $$h(P_0;I):=h(P_0+P').$$
	 By \eqref{equationP1P2}, $h(P_0;I)$ is well defined and for every point $P\in P_0+U(\ell;I)$ we have $$h(P)=h(P_0;I)+(\ell-1)(p-1).$$
	Therefore, 
	we have
\begin{multline}\label{hP}
h(P_0+U(\ell;I))=\#U(\ell;I)\Big(h(P_0;I)+(\ell-1)(p-1)\Big)\\=\ell^{n-\#I}\Big(h(P_0;I)+(\ell-1)(p-1)\Big).
\end{multline}

   For every $\ell\geq 1$ and $k\geq 0$ we put
    $	G_{k,\ell}:= \sum\limits_{i=1}^\ell i^k$, which is well-known to be a polynomial in $\ell$ of degree $k+1$. Therefore, the function 	\begin{multline*}
    \sum_{\ell=1}^{k-1} \ell^{n-\# I}\Big(h(P_0;I)+(\ell-1)(p-1)\Big)
    =G_{n-\#I,k-1}(h(P_0;I)-p+1)+(p-1)G_{n+1-\#I,k-1}
    \end{multline*}
     is a polynomial in $k$ of degree in $n+1$, so is $h(\Delta_k^-)$ when combined with \eqref{hP}.
\end{proof}

\begin{notation}
	We denote by \[\left[ \begin{array}{cccccccccc}
	P_0 & P_1 &\cdots&P_{\ell-1} \\
	Q_0 & Q_1 &\cdots&Q_{\ell-1}  \end{array} \right]_M\]
	the $\ell\times \ell$-submatrix formed by elements of a matrix $M$ whose row indices belong to $\{P_0,P_1,\dots,P_{\ell-1}\}$ and whose column indices belong to $\{Q_0,Q_1,\dots,Q_{\ell-1}\}$.
\end{notation}

Recall that we defined the improved Hodge polygon $\IHP(\Delta)$ in Definition~\ref{IHP}.
\begin{proposition}\label{IHP is below NPchi}
For every $f\in \calF(\Delta)$, the normalized Newton polygon $\NP(f, T)_C$ lies above $\IHP(\Delta)$. 
\end{proposition}

\begin{proof}
	We first recall the definition of $u_{\ell}$ in \eqref{E:Cfstar}, and 
	write $N$ for the standard matrix of $\psi_p\circ E_f$ corresponding to the basis  $\{ \ux^{Q}\}_{Q\in \MM(\Delta)}$ of the Banach space $\mathbf B$. By \cite[Corollary~3.9]{ren-wan-xiao-yu}, we know that the standard matrix of $\psi^n$ corresponding to the same basis is equal to $\sigma_{\Frob}^{m-1}(N)\circ \sigma_{\Frob}^{m-2}(N)\circ\cdots \circ N.$ 
	By \cite[Proposition~4.6]{ren-wan-xiao-yu}, for every $\ell \in \NN$ we have \begin{equation}
	\label{E:expression of char power series}
	\begin{split}
	u_{\ell}(T)
	&=\sum_{\substack{\{Q_{0,0}, Q_{0,1}, \dots, Q_{0,\ell-1}\}\in \mathscr M_\ell\\\{Q_{1,0}, Q_{1,1}, \dots, Q_{1,\ell-1}\}\in \mathscr M_\ell\\\vdots\\\{Q_{{m-1},0}, Q_{{m-1},1}, \dots, Q_{{m-1},\ell-1}\}\in \mathscr M_\ell}} \det\bigg(\prod\limits_{j=0}^{m-1}\left[ \begin{array}{cccccccccc}
	Q_{j+1,0} & Q_{j+1,1} &\cdots&Q_{j+1,\ell-1} \\
	Q_{j,0} & Q_{j,1} &\cdots&Q_{j,\ell-1}  \end{array} \right]_{\sigma_{\Frob}^{j}(N)}\bigg),     \\
	\end{split}
	\end{equation}
	where $Q_{m,i}:=Q_{0,i}$ for each $0\leq i\leq \ell-1$.
	
	We set $\SS_j=\{Q_{j,0}, Q_{j,1}, \dots, Q_{j,\ell-1}\}$, where $\SS_m=\SS_0$, then 
	\begin{equation}\label{explicit of u}
	\begin{split}
	&\val_T\Bigg(\det\Big(\prod\limits_{j=0}^{m-1}\left[ \begin{array}{cccccccccc}
	Q_{j+1,0} & Q_{j+1,1} &\cdots&Q_{j+1,\ell-1} \\
	Q_{j,0} & Q_{j,1} &\cdots&Q_{j,\ell-1}  \end{array} \right]_{\sigma_{\Frob}^{j}(N)}\Big)\Bigg)\\
	=&\val_T\Big(\prod_{j=0}^{m-1}\det\Big(\sigma_{\Frob}^j(e_{pQ'-Q})\Big)_{Q,Q'\in \SS_{j+1}\times\SS_j}\Big)\\
	\geq &\sum_{j=0}^{m-1}\min_{\tau_j\in \Iso(\SS_{j+1},\SS_j)}\Big\{ \sum_{Q\in \SS_{j+1}}\val_T(\sigma_{\Frob}^j(e_{p\tau_j(Q)-Q}))\Big\} \\
			\geq &\sum_{j=0}^{m-1}\min_{\tau_j\in \Iso(\SS_{j+1},\SS_j)}\Big\{ \sum_{Q\in \SS_{j+1}}\lfloor w(p\tau_j(Q))\rfloor-\lfloor w(Q)\rfloor\Big\} \\
	=& \sum_{Q\in \biguplus\limits_{j=0}^{m-1} \SS_{j}}(\lfloor w(pQ)\rfloor-\lfloor w(Q)\rfloor)= h(\biguplus_{j=0}^{m-1} \SS_{j}).
	\end{split}
	\end{equation} 
	
	Therefore, we deduce
	\begin{equation}\label{xk}
	\frac{\val_T(u_{\ell}(T))}{m}\geq \frac{\min\limits_{\SS_0,\dots,\SS_{m-1}\in \mathscr M_\ell} h(\biguplus\limits_{j=0}^{m-1} \SS_{j})}{m}=\min\limits_{\SS\in \mathscr M_\ell}h(\SS)=h(\WW_\ell),
	\end{equation}
which implies that $\NP(T,f)_C$ lies on or above $\IHP(\Delta)$.
\end{proof}
\begin{corollary}\label{corollary specialization}
	Let $\chi:\ZZ_p\to \CC_p^\times$ be a finite character of conductor $p^{m_\chi}$.  
	The improved Hodge polygon is a lower bound of the normalized Newton polygon $\NP(f, \chi)_C$.

\end{corollary}
\begin{proof}
Note that specializing $T$ with $\chi(1)-1$ does not lower the polygon. Therefore, for every nontrivial finite character $\chi$, the Newton polygon $\NP(f, \chi)_C$ lies on or above $\NP(f, T)_C$. Hence, $\NP(f, \chi)_C$ lies on or above $\IHP(\Delta)$, when combined with Proposition~\ref{IHP is below NPchi}. 
\end{proof}
\begin{corollary}\label{first corollary}
For every $k\geq 1$ we have $$u_{\mathbbm x_k^{\pm}}\equiv \prod\limits_{j=0}^{m-1}\sigma_{\Frob}^j\Big(\det(e_{pQ'-Q})_{Q,Q'\in \Delta_k^\pm} \Big) \pmod {T^{mh(\Delta^\pm_k)+1}}.$$
\end{corollary}
\begin{proof}
 By Proposition~\ref{corollary smallest set}, when $\ell=\mathbbm x_k^{\pm}$, the equalities hold in \eqref{explicit of u} if and only if \[\SS_j=\Delta^\pm_k\ \textrm{for every}\ 0\leq j\leq m-1.\qedhere\]
\end{proof}

	 \section{The generic Newton polygon}\label{section 4} 
	 In this section, we will prove Theorem~\ref{generic newton polygon}, which is given by a sequence of sufficient statements.
	 
%	 Most results in this section apply for all $1\leq \ell\leq n$. Therefore, without additional instruction, we put $\Delta'$ to be an $\ell$-paralleltope in $\RR^n$ as $$\Delta'=\Big\{\sum_{i=1}^{\ell}z_i'V_i'\;\big|\;0\leq i\leq 1\Big\},$$
%	 where $\overrightarrow{\calO V'_1},\overrightarrow{\calO V'_2},\dots,\overrightarrow{\calO V'_{\ell}}$ are linearly independent.
	 
	 Recall that $\widetilde e_{Q}(T)$ is defined in \eqref{equation: E Delta}, and satisfies $\val_T(\widetilde{e}_{Q}(T))\geq \lfloor w(pQ)\rfloor-\lfloor w(Q)\rfloor$ for all $Q \in \MM(\Delta)$ (see Lemma~\ref{L:estimate of Ef(x)} (2)).
	 
	 \begin{notation}\label{definition of u}
	 	We put \begin{equation}\label{tildeu}
	 	\det(I-\widetilde{e}_{pQ'-Q}s)_{Q,Q'\in \MM(\Delta)}=\sum_{i=0}^{\infty}\widetilde{u}_i(T)s^i\in \Zp[\widetilde a_P; P\in \Delta^+\backslash \{\calO\}][\![s]\!].
	 	\end{equation}

	 \end{notation}
	 Similar to \eqref{xk}, we have 
	 	\begin{equation}\label{equation xk}
	 \val_T(\widetilde u_{\mathbbm x_k^\pm}(T))\geq h(\Delta^\pm_k),
	 \end{equation}
	 which allows us to put
	 \begin{equation}\label{tildeui}
	 \widetilde{u}_{\mathbbm x_k^\pm}(T):=\sum_{i=h(\Delta_k^\pm)}^{\infty} \widetilde{u}_{\mathbbm x_k^\pm,i}T^{i}\quad \textrm{for}\  \widetilde{u}_{\mathbbm x_k^\pm,i}\in \Zp[\widetilde a_P; P\in \Delta^+\backslash \{\calO\}]
	 \end{equation}
	 with an explicit lower bound of the summation.
	 
%	  \begin{proposition}\label{proposition}
%	 	There are a Zariski open subset $O_\Zar\subseteq  (\overline\FF_p)^{\mathbbm x'_1}$ such that \\
%	 	if $$(a_{P})_{P\in \Delta^+}\in O_{\Zar}\ \text{and}\ f(\ux)=\sum\limits_{P\in \Delta^+}a_{P}\ux^P,$$ then $$\widetilde u_{\mathbbm x_k}|_{\widetilde{a}_P=a_{P}}\not\equiv 0 \pmod{(p, T^{h(\overline\Delta^\Int_k)+1})}\ \textrm{for every}\ 1\leq k\leq n-1.$$
%	 \end{proposition}

%We will show in Proposition~\ref{proposition 1} that the following proposition leads to Theorem~\ref{generic newton polygon}, and prove Theorem~\ref{theorem for L} as a corollary.

\begin{proposition}\label{proposition}
For every integer $1\leq k\leq n+2$, we have
 \begin{equation}\label{tilde u}
 \widetilde u_{\mathbbm x_k^\pm,h(\Delta_k^\pm)}\not\equiv 0 \pmod{p}.
 \end{equation}
\end{proposition}

We will give its proof in \S\ref{section 5}.

	 \begin{proposition}\label{proposition 1}
	 	Proposition~\ref{proposition} implies Theorems~\ref{generic newton polygon} and ~\ref{theorem for L}.
	 \end{proposition}
% The proof of Proposition~\ref{proposition 1} is given after several lemmas.

% 
%We first fix some character $\chi_0:=\Zp\to\CC_p^\times$ of conductor $p$, and denote the Newton slopes of $\NP_L(f,\chi)$ by $$\Big(\alpha_1,\alpha_2,\dots,\alpha_{n!\Vol(\Delta)}\Big)$$ in a non-descending order and put $$X_k:=\sum\limits_{j=1}^{\mathbbm{t}_k}\alpha_j,$$
% where $\mathbbm t_k:=\sum\limits_{i=0}^k (-1)^i \binom{n}{i} \mathbbm x_{k-i}$. 
% We denote 	$$C_f^*(\chi,s) := \sum_{k=0}^\infty u_{\chi,k} s^k \in \ZZ_p\llbracket T. s\rrbracket.$$
% 
% 

% 	
% 	,Assume that for each $0\leq k\leq n$, the Newton slopes of $\NP_L(f,\chi)$ satisfies
% 	\begin{equation}\label{equation Newton slopes}
% 	\begin{cases}
% 	\alpha_i <k(p-1)& \textrm{if}\ i\leq \mathbbm t_k,\\
% 	\alpha_i \geq k(p-1)& \textrm{if}\ i> \mathbbm t_k.
% 	\end{cases}
% 	\end{equation}
%Then for every $\ell\geq n$, we have $$u_{\chi,\mathbbm x_\ell}=\sum_{i=0}^{n+1}B_i \ell^i.$$

% \begin{proof}
%  Consider the relation between the characteristic function and $L$-function associated to $\chi$ and $f$, we have
% \begin{equation}\label{equation CL}
% C^*_f(\chi, s)^{(-1)^{n-1}}=\Big(\prod\limits_{j=0}^{\infty}L^*_f(T, q^js)^{\binom{n+j-1}{n-1}}\Big).
% \end{equation}
% Therefore, there are exactly   
% $$\sum_{j=\ell-n+1}^{\ell-1}\binom{n+j-1}{j}X_{\ell-j}+\sum_{j=1}^{\ell-n}\binom{n+j-1}{n-1}n!\Vol(\Delta)$$
% segment in $\NP_C(\chi, f)$ with slopes strictly less than $\ell$.
%
% 
% \end{proof}
% 

% It is easy to show that \ref{equation Newton slopes} follows directly from \eqref{equation CL} and \ref{property (b)}(b) above.

We give its proof after several lemmas.
\begin{lemma}\label{lemmafor1.7(1)}
	Assume Proposition~\ref{proposition}. Then there exists a Zariski open subset $O_\Zar\subseteq \calF(\Delta)$ such that for every $f\in O_\Zar$ and every finite character $\chi_0$ of conductor $p$, the Newton polygon $\NP(f,\chi_0)_C$ passes though the point $(\mathbbm x_k^\pm,h(\Delta_k^\pm))$ for every $1\leq k\leq n+2$.
\end{lemma}

 \begin{proof}
%	By Corollary~\ref{corollary specialization}, for every $f\in \mathcal F(\Delta)$ and every finite character $\chi:\ZZ_p\to\CC_p^\times$, the Newton polygon $\NP(f,\chi)_L$ is not below 
% 	$(\mathbbm x_k^\pm,h(\Delta_k^\pm))$ for $k\geq 1$.
% 	
% 	Therefore, we are left to show that for every $k\geq0$ there exists a $f\in \mathcal F(\Delta)$ such that $\NP(f,\chi)_L$ passes though $(\mathbbm x_k^\pm,h(\Delta_k^\pm))$.
% 	To this end, w
Since $\widetilde u_{\mathbbm x_k^\pm,h(\Delta_k^\pm)}\not\equiv 0 \pmod{p}$ for every $1\leq k\leq n+2,$
the subset
 	\begin{multline}\label{Ozar}
 	O_\Zar(\Delta):=
 	\Big\{f(\ux)=\sum_{P\in \Delta^+}a_P\ux^P\in \calF(\Delta):
 	 \widetilde u_{\mathbbm x^\pm_k,h(\Delta_k^\pm)}\Big|_{\widetilde{a}_P={\hat a}_{P}}\not\equiv 0 \pmod p\\ \textrm{for every}\ 1\leq k\leq n+2\Big\}
 	\end{multline}
 	 is a nonempty open subset of $\calF(\Delta)$.

 	Similar to Corollary~\ref{first corollary},  we have
 	\begin{equation}\label{eqQQ}
 	\det\left(\widetilde{e}_{pQ'-Q}\right)_{Q,Q'\in \Delta_k^\pm}\equiv\widetilde u_{\mathbbm x_k^\pm,h(\Delta_k^\pm)} T^{h(\Delta_k^\pm)}\mod (p, T^{h(\Delta_k^\pm)+1}).
 	\end{equation}

% 	such that $$\widetilde u_{\mathbbm x_k}|_{\widetilde{a}_P=\hat a_{P}}\not\equiv 0 \pmod{(p, T^{h(\overline\Delta^\Int_k)+1})}\ \textrm{for every}\ 1\leq k\leq n-1.$$
% 	$(a_{P})_{P\in \Delta^+}\in O_{\Zar}\ \text{and}\$
% 	
% 	there is  such that 
% 	if $$(a_{P})_{P\in \Delta^+}\in O_{\Zar}\ \text{and}\ f(\ux)=\sum\limits_{P\in \Delta^+}a_{P}\ux^P,$$ then $$\widetilde u_{\mathbbm x_k}|_{\widetilde{a}_P=a_{P}}\not\equiv 0 \pmod{(p, T^{h(\overline\Delta^\Int_k)+1})}\ \textrm{for every}\ 1\leq k\leq n-1.$$
% 	
% 	 $\widetilde{e}_{P}|_{\widetilde{a}_P=a_{P}}= e_{P}$, for every $k\leq n-1$ we have 
	
	Hence, for every $f(\ux)=\sum\limits_{P\in \Delta^+}a_P \ux^P\in O_\Zar(\Delta)$ and every $ 1\leq k\leq n+2$ we have 
\begin{multline}\label{equation for estimating u}
u_{\mathbbm x_k^\pm}\equiv \prod\limits_{j=0}^{m(f)-1}\sigma_{\Frob}^j\Big(\det(e_{pQ'-Q})_{Q,Q'\in \Delta_k^\pm} \Big)
=\prod\limits_{j=0}^{m(f)-1}\sigma_{\Frob}^j\Big(\det(\widetilde{e}_{pQ'-Q})_{Q,Q'\in \Delta_k^\pm}\Big|_{\widetilde{a}_{P}=\hat a_{P}} \Big)\\
\equiv \prod\limits_{j=0}^{m(f)-1}\sigma_{\Frob}^j\left(\widetilde u_{\mathbbm x_k,h(\Delta_k^\pm)}\Big|_{\widetilde{a}_{P}=\hat a_{P}}\right)T^{m(f)h(\Delta^\pm_k)}
\not\equiv 0 
\pmod {p,T^{m(f)h(\Delta^\pm_k)+1}},
\end{multline}
where $m(f)=[\FF_p(f):\FF_p].$

Combining \eqref{specialization} with Corollary~\ref{corollary specialization}, we get $$C_f^*(\chi_0,s)=C_f^*(T,s)\Big|_{T=\chi_0(1)-1}\quad \textrm{and}\quad u_{\mathbbm x_k^\pm}=\sum_{\ell=T^{mh(\Delta^\pm_k)}}^\infty u_{\mathbbm x_k^\pm,\ell}T^\ell. $$ 
Together with  \eqref{equation for estimating u},  these two equalities imply that 
for every $1\leq k\leq n+2$ the $p$-adic valuation of the coefficient of $s^{\mathbbm x_k^\pm}$ in $C_f^*(\chi_0,s)$ satisfies 
$$\val_p\left(u_{\mathbbm x_k^\pm}\Big|_{T=\chi_0(1)-1}\right)=m\val_p(\xi_p-1)h(\Delta^\pm)=\frac{mh(\Delta_k^\pm)}{p-1}.$$
It implies that $\NP(f,\chi_0)_C$ lies below $(\mathbbm x_k^\pm,h(\Delta_k^\pm))$ for every $1\leq k\leq n+2$.
Since $(\mathbbm x_k^\pm,h(\Delta_k^\pm))$ are \emph{vertices} of $\IHP(\Delta)$ which, by Corollary~\ref{corollary specialization}, is a lower bound for $\NP(f,\chi_0)_C$, we conclude that $\NP(f,\chi_0)_C$ must pass though the points $(\mathbbm x_k^\pm,h(\Delta_k^\pm))$ for every $1\leq k\leq n+2$.
\end{proof}

	\begin{lemma}\label{lemmafor1.7(2)}
		For a polynomial $f\in \mathcal F(\Delta)$ and a finite character $\chi_0$ of conductor $p$, if $\NP(f,\chi_0)_C$ passes though $(\mathbbm x_k^\pm,h(\Delta_k^\pm))$ for some $1\leq k\leq n+2$, then it passes
		$(\mathbbm x_k^\pm,h(\Delta_k^\pm))$ for all $k\geq 1$. 
	\end{lemma}
\begin{proof}
	Since for every $1\leq k\leq n+2$ the Newton polygon $\NP(f,\chi_0)_C$ passes though the points $(\mathbbm x_k^-,h(\Delta_k^-))$  and $(\mathbbm x_k^+,h(\Delta_k^+))$, which are vertices of its lower bound $\IHP(\Delta)$, we have the following.
		\begin{itemize}
		\item[(a)] The points $(\mathbbm x_k^-,h(\Delta_k^-))$ and $(\mathbbm x_k^+,h(\Delta_k^+))$ are vertices of $\NP(f,\chi_0)_C$, and the segment connecting them is contained in $\NP(f,\chi_0)_C$.
		\item[(b)]\label{property (b)}  There are $\mathbbm x_k^-$ slopes of $\NP(f,\chi_0)_C$ strictly less than $k(p-1)$.
		\item[(c)]\label{property (c)}  There are $\mathbbm x_k^+$ slopes of $\NP(f,\chi_0)_C$ less than or equal to $k(p-1)$.
	\end{itemize}

We denote by
$\LL^\star:=\{\alpha_1,\alpha_2,\dots,\alpha_{n!\Vol(\Delta)}\}$ the Newton slopes of $\NP(f,\chi_0)_L$ in a non-descending order. By Weil conjecture, $\alpha\in [0,n(p-1)]$ for every $\alpha\in \LL^\star$.
We denote by $\mathbbm t_i$ the index such that 
$$\mathbbm t_i=\#\{\alpha\in \LL^\star\;|\;\alpha< i(p-1)\}
\quad\textrm{for}\quad0\leq i\leq n.$$
%In particular, we write $\mathbbm t_0$
%for the number of zero slopes in $\LL$.

Set $$\mathbf X_i:=\sum\limits_{j=1}^{\mathbbm{t}_i}\alpha_j\quad\textrm{and}\quad C_f^*(\chi_0,s) := \sum_{i=0}^\infty u_{\chi_0,i} s^i \in \ZZ_p\llbracket s\rrbracket.$$
Consider the relation between $C^*_f(\chi_0, s) $ and $L^*_f(\chi_0, s)$:
\begin{equation}
C^*_f(\chi_0, s)=\Big(\prod\limits_{j=0}^{\infty}L^*_f(\chi_0, q^js)^{\tbinom{n+j-1}{n-1}}\Big)^{(-1)^{n-1}}.
\end{equation}
We obtain that for every $\ell\geq 1$ there are
$$\sum_{j=\ell-n}^{\ell-1} \tbinom{n+j-1}{n-1}( \mathbbm t_{\ell-j+1}-\mathbbm t_{\ell-j})$$
slopes of $\NP(f,\chi_0)_C$  are contained in $[(\ell-1)(p-1),\ell(p-1))$,
where $\mathbbm t_{<0}=0$. 
Therefore, the slopes of $\NP(f,\chi_0)_C$ strictly less than $k(p-1)$ is equal to 
$$\sum_{\ell=1}^{k} \sum_{j=\ell-n}^{\ell-1} \tbinom{n+j-1}{n-1}( \mathbbm t_{\ell-j+1}-\mathbbm t_{\ell-j}),$$
and $\NP(f,\chi_0)_C$ passes through the points 
$$\Big(\sum_{\ell=1}^{k} \sum_{j=\ell-n}^{\ell-1} \tbinom{n+j-1}{n-1}( \mathbbm t_{\ell-j+1}-\mathbbm t_{\ell-j}),\sum_{\ell=1}^{k} \sum_{j=\ell-n}^{\ell-1} \tbinom{n+j-1}{n-1}( \mathbf X_{\ell-j+1}-\mathbf X_{\ell-j}+j(p-1)(\mathbbm t_{\ell-j+1}-\mathbbm t_{\ell-j})\Big)$$
for all $k\geq 1$.

Combining it with (b), for every $1\leq k\leq n+2$ we have 
\begin{equation}\label{equation xk=}
\sum_{\ell=1}^{k} \sum_{j=\ell-n}^{\ell-1} \tbinom{n+j-1}{n-1}( \mathbbm t_{\ell-j+1}-\mathbbm t_{\ell-j})=\mathbbm x_k^-=k^n\Vol(\Delta).
\end{equation}

Note that the left hand side of \eqref{equation xk=} is equal to \[
\sum_{i=1}^{n}\Big(( \mathbbm t_{i+1}-\mathbbm t_{i})\sum_{\ell=1}^{k}  \tbinom{n+\ell-i-1}{n-1}\Big),
\]
which is a polynomial of degree $n$.
%
%Since  for every $1\leq i\leq n$, the sum $$(\mathbbm t_{i+1}-\mathbbm t_{i})\sum_{\ell=1}^{k}  \binom{n+\ell-i-1}{n-1}$$ is a polynomial in $k$ of degree $n$, we know that $$\sum_{\ell=1}^{k} \sum_{j=\ell-n}^{\ell-1} \binom{n+j-1}{n-1}( \mathbbm t_{\ell-j+1}-\mathbbm t_{\ell-j})$$ is also a polynomial in $k$ of degree $n$.  
Since the both sides of \eqref{equation xk=} are polynomials in $k$ of degree $n$ and they have the same values for every $1\leq k\leq n+2$, they must be identical as polynomials. Namely, the equality \eqref{equation xk=} holds for all $k\geq 1$.

On the other hand, combining \eqref{equation xk=} with (a), we get
\begin{equation}\label{equation for Xk}
\sum_{\ell=1}^{k} \sum_{j=\ell-n}^{\ell-1} \tbinom{n+j-1}{n-1}( \mathbf X_{\ell-j+1}-\mathbf X_{\ell-j}+j(p-1)(\mathbbm t_{\ell-j+1}-\mathbbm t_{\ell-j}))=h(\Delta^-_k)
\end{equation}
for every $1\leq k\leq n+2$.

Clearly, the left hand side of \eqref{equation for Xk} is a polynomial in $k$ of degree $n+1$. By Lemma~\ref{lemma polynomial for h}, the right hand side of \eqref{equation for Xk} is also a polynomial in $k$ of degree $n+1$.
Running the similar argument as above, we conclude that the equality \eqref{equation for Xk} holds for every $k\geq 1$.
Therefore, the Newton polygon $\NP(f,\chi_0)_C$ passes through the point $(\mathbbm x_k^-,h(\Delta_k^-))$ for every $k\geq 1$. 

A similar argument shows that $\NP(f,\chi_0)_C$ passes through the point $(\mathbbm x_k^+,h(\Delta_k^+))$ for every $k\geq 1$. 
\end{proof}
	\begin{lemma}\label{lemmafor1.7(3)}
	For a polynomial $f\in \mathcal F(\Delta)$ with $m=[\FF_p(f):\FF_p]$, if there exists a finite character $\chi_0$ of conductor $p$ such that
	$\NP(f,\chi_0)_C$ passes though $(\mathbbm x_k^\pm,h(\Delta_k^\pm))$ for all integer $k\geq 1$, then for every nontrivial finite character $\chi$  the Newton polygon $\NP(f,\chi)_C$ passes though the points $(\mathbbm x_k^\pm,h(\Delta_k^\pm))$ for all $k\geq 0$.
\end{lemma}

\begin{proof}
	Recall that $C^*_f(T,s)=\sum\limits_{i=0}^\infty u_is^i.$
Since $$u_{\mathbbm x_0^-}=1\quad\textrm{and}\quad u_{\mathbbm x_0^+}=-\textrm{Norm}_{\FF_p(f)/\FF_p}(E(\hat a_\calO\pi))\equiv -1\pmod T,$$
	the Newton polygon $\NP(f,\chi)_C$ passes through the point $(\mathbbm x_0^\pm,h(\Delta_0^\pm))$.
	
	We  next show that for every integer $k\geq 1$ we have \begin{equation}\label{4.10}
	u_{\mathbbm x_k^\pm}\not\equiv 0\pmod {p,T^{mh(\Delta^\pm_k)+1}}.
	\end{equation}
	
	Suppose it is false. Without loss of generality, we may assume that there exists $k_0\geq 1$ such that $$u_{\mathbbm x_{k_0}^-}\equiv 0\pmod {p,T^{mh(\Delta^-_{k_0})+1}}.$$
	This congruence equality implies that the $p$-adic valuation of the coefficient of $s^{\mathbbm x_{k_0}^-}$ in $C_f^*(\chi_0,s)$ satisfies 
	$$\val_p\left(u_{\mathbbm x_{k_0}^-}\Big|_{T=\chi_0(1)-1}\right)>\frac{mh(\Delta_{k_0}^-)}{p-1},$$
and hence the Newton polygon $\NP(f,\chi_0)_C$ can never pass through the point $(\mathbbm x_{k_0}^-,h(\Delta_{k_0}^-))$. However, $(\mathbbm x_{k_0}^-,h(\Delta_{k_0}^-))$ is a vertex of the lower bound $\IHP(\Delta)$ of $\NP(f,\chi_0)_C$, a contradiction to the assumption in this lemma.
	
Therefore, we prove the inequality \eqref{4.10} and hence
$$\val_p\left(u_{\mathbbm x_k^\pm}\Big|_{T=\chi(1)-1}\right)= \frac{mh(\Delta^\pm_k)}{(p-1)p^{m_\chi-1}}.
$$
The last equality implies that $\NP(f,\chi)_C$ passes through $(\mathbbm x_k^\pm,h(\Delta_k^\pm))$ for every $k\geq 1$ when combined with Corollary~\ref{corollary specialization}.
\end{proof}
%On the other hand, it is easy to check by definition that $$\widetilde u_{\mathbbm x_0^\pm,h(\Delta_0^\pm)}=\mp 1.$$
%
%As a result, for every character $\chi:\ZZ_p\to \CC_p^\times$ of conductor $m(\chi)$, the Newton polygon $\NP(f,\chi)_L$ passes though $(\mathbbm x_k^\pm,h(\Delta_k^\pm))$ for every $0\leq  k\leq p^{m(\chi)-1}$, which completes the proof.

\begin{proof}[\textbf{Proof of Theorem~\ref{generic newton polygon}}]
	
	We know already from Corollary~\ref{corollary specialization} that $\IHP(\Delta)$ is a lower bound for $\GNP(\Delta)$.  Now we show that it is sharp at the point $(\mathbbm x_{k}^\pm,h(\Delta_{k}^\pm))$ for every $k\geq 0$.
	
	If we choose a polynomial $f\in O_\Zar$, by Lemmas~\ref{lemmafor1.7(1)} and \ref{lemmafor1.7(2)}, for every finite character $\chi_0$ of conductor $p^{m_{\chi_0}}$ and every integer $k\geq 0$, the Newton polygon $\NP(f,\chi_0)_C$ passes though  the point $(\mathbbm x_k^\pm,h(\Delta_k^\pm))$.
\end{proof}
	\begin{proof}[\textbf{Proof of Theorem~\ref{theorem for L}}]
		
		Let $\chi$ be a nontrivial finite character and $f\in O_\Zar$, where $O_\Zar$ is defined as in Lemma~\ref{lemmafor1.7(1)}. By  Lemmas~\ref{lemmafor1.7(1)}, ~\ref{lemmafor1.7(2)} and ~\ref{lemmafor1.7(3)}, the Newton polygon $\NP(f,\chi)_C$ passes through the point $(\mathbbm x_k^\pm,h(\Delta_k^\pm))$  for every $k\geq 0$. Since these points are vertices of $\IHP(\Delta)$ which is a lower bound of $\NP(f,\chi)_C$, they must also be vertices of $\NP(f,\chi)_C$.

		Therefore, by \eqref{equation LC}, for every $0\leq i_1\leq n-1$ and every $0\leq i_2\leq p^{m_\chi-1}-1$ there are $$\sum\limits_{t=0}^{i_1}(-1)^i\binom{n}{i}\mathbbm x^+_{(i_1-t)p^{m_\chi-1}+i_2}\  \left(\mathrm{resp.}\ \sum\limits_{t=0}^{i_1}(-1)^i\binom{n}{i}\mathbbm x^-_{(i_1-t)p^{m_\chi-1}+i_2}\right)$$ slopes of $\NP(f,\chi)_L$ less than or equal to (resp. strictly less than) $i_1(p-1)p^{m_\chi-1}+(p-1)i_2.$
		Since the slopes of $\NP(f,\chi)_L$ divided by $(p-1)p^{m_\chi}$ are the $p^{m(f)}$-adic Newton slopes of $L^*_f(\chi,s)$, we complete the proof.
	\end{proof}

	 \section{Nonvanishing of leading terms in universal coefficients}\label{section 5}

To show \eqref{tilde u} for $1\leq k\leq n+2$,  it is enough to show the coefficient of one special term in $\widetilde u_{\mathbbm x_k^\pm}$ is not zero. Computing explicitly the contribution of this permutation of $\Delta_k^{\pm}$ to this special term, which itself is subdivided into simpler cases, allows us to prove \eqref{tilde u}.
	 \subsection{Proof of  Proposition~\ref{proposition} with assuming Propositions~\ref{lemma section 5 (1)} and \ref{lemma section 5 (2)}.}
	 
	 	Recall that in Notation~\ref{eta}, we define the function $\eta$ which maps $Q\in \Delta^-$ to the residue of $pQ$ module $\Lambda_{\Delta}$ in $\Delta^-$.
	 
%	Recall that we define the determinant of a bi-variable function in Notation~\ref{det}
%	\begin{lemma}\label{lemma to e}
%		We have $$\widetilde{u}_{\mathbbm x_k^\pm}\equiv\det(\widetilde e_{pQ-Q'})_{Q,Q'\in \Delta_k^\pm}\pmod{T^{h(\Delta_k^\pm)+1}}.$$
%	\end{lemma}
%\begin{proof}
%	It follows from Lemma~\ref{L:estimate of Ef(x)} and \ref{first lemma}.
%\end{proof}
  	 \begin{notation}\label{Delta(I)}\hfill
 	
 	%	 \begin{definition}
 	%	 	For a power series $g(s)=\sum\limits_{i=0}^{\infty}b_{i}s^i$ in $s$, we put $\LD(g,s)s^{\deg(g,s)}$ to be the leading term of $g$, where $\ord(g,s)$ is its degree. 
 	%	 \end{definition}

 	\begin{enumerate}
 		\item For every non-empty subset $S\subseteq \{1,2,\dots,n\}$, we set $\mathbf V_S:=\sum\limits_{i\in S} \mathbf V_i$, and denote by 
$$\Ver(\Delta):=\{\mathbf V_S\;\big|\;S\subseteq \{1,2,\dots,n\}\ \textrm{and}\ S\neq \emptyset\}$$ the set consisting of all vertices of $\Delta$ except the origin $\calO$. 
 		\item For a polynomial $g(\widetilde a_{P})\in \ZZ_p[\widetilde a_{P},P\in \Delta^+\backslash \{\calO\}]$, we denote by  $g^{\res}$ the polynomial in $\ZZ_p[\zeta_1, \zeta_2, \dots, \zeta_n]$ obtained by \begin{itemize}
 			\item specializing $\widetilde a_{P}$ with $0$ if $P\notin \Ver(\Delta)$, and
 			\item replacing $\widetilde a_{P}$ by $\zeta_{\#S}$ if $P=\mathbf V_S$.
 		\end{itemize}
 		\item For a point $P_0\in \Delta^-$, we put $$\Delta_k^\pm(P_0)=\big\{Q\in\Delta_k^\pm\;\big|\;Q\equiv P_0 \pmod {\Lambda_{\Delta}}
 		\big\}.$$
 	\end{enumerate}
 \end{notation}

\begin{theorem}\label{lemma for res}
	For every $1\leq k\leq n+2$, we have 
	\begin{equation}\label{equation for res}
	\det\left(\widetilde e^\res_{pQ-Q'}\right)_{Q,Q'\in \Delta_k^\pm}\not\equiv 0\mod (p, T^{h(\Delta^\pm_k)+1}).
	\end{equation}
\end{theorem}
The rest of the paper is devoted to the proof of this theorem.

\begin{lemma}
Theorem~\ref{lemma for res} implies  Proposition~\ref{proposition}.
\end{lemma}
\begin{proof}
From \eqref{equation for res} we have
$$\det(\widetilde e_{pQ'-Q})_{Q,Q'\in \Delta_k^\pm}\not\equiv 0\mod (p, T^{h(\Delta^\pm_k)+1})$$ for every $1\leq k\leq n+2$. Combined with \eqref{eqQQ}, these congruence inequalities imply that
$$\widetilde u_{\mathbbm x_k^\pm,h(\Delta_k^\pm)}\not\equiv 0 \pmod{p}$$
for every $1\leq k\leq n+2$. 
\end{proof}

	\begin{lemma}\label{lemma first}
	For every $\tau\in \Iso(\Delta_k^\pm)$, if there exist $P_0\in \Delta^-$ and $Q_0\in \Delta_k^\pm(P_0)$ such that $\tau(Q_0)\notin \Delta_k^\pm(\eta(P_0))$, then 
	\begin{equation}\label{prodres}
	\prod_{Q\in \Delta_k^\pm(P_0)}\widetilde{e}^\res_{pQ-\tau(Q)}=0.
\end{equation}
\end{lemma}
\begin{proof}
		Recall that for every $Q\in \MM(\Delta)$, we have 	\begin{equation}\label{ep}
	\widetilde e_{Q}(T)=\sum_{\Big\{\vec{j}\in \ZZ_{\geq 0}^{\Delta^+\backslash \{\calO\}}\;\Big|\;\sum\limits_{P\in \Delta^+\backslash \{\calO\}}j_PP=Q\Big\}}\Big(\prod_{P\in \Delta^+\backslash \{\calO\}}c_{j_P} (\widetilde{a}_{P} \pi)^{j_P}\Big). 
	\end{equation}
	
	If $\tau(Q_0)\notin \Delta_k^\pm(\eta(P_0))$, then 
	$pQ_0-\tau(Q_0)\not\equiv \calO \mod \Lambda_\Delta$. Therefore, every linear combination $$pQ_0-\tau(Q_0)=\sum_{P\in \Delta^+\backslash \{\calO\}} j_PP$$ contains $P\notin \Ver(\Delta)$ must have $j_P\neq 0$, which implies \[\widetilde{e}^\res_{pQ_0-\tau(Q_0)}=0.\]
	Hence, to compute 	$\det(\widetilde e^\res_{pQ-Q'})_{Q,Q'\in \Delta_k^\pm}$, it suffices to take sum over all the permutations $\tau$ such that $ \tau(Q)\equiv pQ\pmod{\Lambda_{\Delta}}$
for every $Q\in \Delta_k^\pm$.
\end{proof}
%	\begin{lemma}\label{reducetores}
%		
%		The determinant
%	
%			\begin{equation}
%		\det(\widetilde e^\res_{pQ-Q'})_{Q,Q'\in \Delta_k^\pm}=\prod_{P_0\in \Delta^-}\det(\widetilde e^\res_{P_0-\eta(P_0)+pQ-Q'})_{Q,Q'\in\Delta_k^\pm(P_0)}.
%		\end{equation}
%	\end{lemma}
%\begin{proof}
%	Lemma~\ref{lemma first} shows that 
%		\begin{equation*}
%		\begin{split}
%	\det(\widetilde e^\res_{pQ-Q'})_{Q,Q'\in \Delta_k^\pm}=&\prod_{P_0\in \Delta^-}\det(\widetilde e^\res_{pQ-Q'})_{Q\in\Delta_k^\pm(P_0),Q'\in\Delta_k^\pm(\eta(P_0))}\\
%	=&\prod_{P_0\in \Delta^-}\det(\widetilde e^\res_{P_0-\eta(P_0)+pQ-Q'})_{Q,Q'\in\Delta_k^\pm(P_0)}.\qedhere
%		\end{split}
%	\end{equation*}
%\end{proof}
%	 	\begin{proof}
%	 		Let $P=\sum\limits_{i=1}^{n}r_i\mathbf{V}_i$ be a point in $\Delta$ such that $P\%\neq \calO$. By Cramer's rule, we know that $0\leq r_i<1$ and $r_i\Vol(\Delta)\in \ZZ$. Since $p$ and $\Vol(\Delta)$ are coprime, there exists an $i$ such that $pr_i\notin \ZZ$. Hence, we have $(pP)\%\neq \calO$, which shows that $\eta$ is injective. Therefore, it is also a bijection. 
%	 	\end{proof}

	 \begin{notation}
	 For a point $Q=\sum\limits_{i=1}^n z_i\mathbf V_i$,  we set $$I(Q):=\{i\;\big|\; z_i=0\}.$$
	For a subset $I\subseteq\{1,2,\dots,n\}$, we generalize Notations~\ref{notation 3.14} and~\ref{Delta(I)}(3) to
\begin{itemize}
	\item 	$\Delta_k^\pm(I):=\Big\{Q=\sum\limits_{i=1}^n z_i\mathbf V_i\in \Delta_k^\pm\Bigg|\;
	\begin{aligned}
	&	z_i=0\quad  \textrm{if}\ i\in I\\
	&	z_i>0 \quad  \textrm{otherwise}
	\end{aligned}
	\Big\},$
	\item $\Delta_k^+(I,P_0):=\Delta_k^+(I)\cap \Delta_k^+(P_0)\quad\textrm{and}\quad \Delta_k^-(I,P_0):=\Delta_k^-(I)\cap \Delta_k^-(P_0).$
\end{itemize}
\end{notation}
\begin{lemma}\hfill
	\begin{enumerate}
		\item For each $P_0\in \Delta^-$, if $I\not\subseteq I(P_0)$, then $\Delta_k^\pm(I,P_0)=\emptyset$.
	\item The set $\Delta_k^\pm$ is a disjoint union of $\Delta_k^\pm(P_0,I)$ for all points $P_0\in \Delta^-$ and all subsets $I\subseteq I(P_0)$.
	\end{enumerate}
\end{lemma}
\begin{proof}
Both (1) and (2) are straightforward.
\end{proof}
\begin{lemma}\label{lemma second}
		For every $\tau\in \Iso(\Delta_k^\pm)$, if there exists a subset $I\subseteq \{1,2,\dots,n\}$ such that $\tau(\Delta_k^\pm(I))\neq \Delta_k^\pm(I) $, then \begin{equation}\label{prodsurface}
		\prod_{Q\in \Delta_k^\pm(P_0)}\widetilde{e}_{pQ-\tau(Q)}=0.
		\end{equation}
\end{lemma}
\begin{proof}
		
	We put $I_0\subseteq \{1,2,\dots,n\}$ to be one of the smallest subset such that $$\tau( \Delta_k^\pm(I_0))\neq \Delta_k^\pm(I_0).$$ Namely,  every set $I$ that contains fewer elements than $I_0$ satisfies
	$$\tau( \Delta_k^\pm(I))= \Delta_k^\pm(I).$$
	
	It implies a point $Q_0\in \Delta_k^\pm(I_0)$ 
	such that $pQ_0-\tau(Q_0)\notin \Cone(\Delta)$. Hence, we obtain $\widetilde e_{pQ_0-\tau(Q_0)}=0$
	and consequently \eqref{prodsurface}.
\end{proof}

	 	\begin{proposition}\label{proposition first}
	 	We have the following equality
 		\begin{equation}\label{equation first}
 		\det\left(\widetilde e^\res_{pQ-Q'}\right)_{Q,Q'\in \Delta_k^\pm}=\prod_{P_0\in \Delta^-}\prod_{I \subseteq I(P_0)}\det\left(\widetilde e^\res_{P_0-\eta(P_0)+pQ-Q'}\right)_{Q,Q'\in\Delta_k^\pm(I,P_0)}.
 		\end{equation}
	 	\end{proposition}
	 	\begin{proof}
	 	By Lemmas~\ref{lemma first} and~\ref{lemma second}, for every $\tau\in \Iso(\Delta_k^\pm)$ if \[\prod_{Q\in \Delta^\pm(P_0)}\widetilde e^\res_{pQ-\tau(Q)}\neq 0,\] then $\tau$ must map $\Delta_k^\pm(I,P_0)$ onto $\Delta_k^\pm(I,\eta(P_0))$ for every  point $P_0\in \Delta^-$ and every integer $I\in I(P_0)$. 
	 	Note that
	 	\begin{equation}\label{p0}
	 	\Delta_k^\pm(I,\eta(P_0))=\Delta_k^\pm(I,P_0)-P_0+\eta(P_0).
	 	\end{equation}
	 	Then we have 
	 	\begin{multline*}
	 	\det\left(\widetilde e^\res_{pQ-Q'}\right)_{Q,Q'\in \Delta_k^\pm}=\prod_{P_0\in \Delta^-}\prod_{I \subseteq I(P_0)}\det\left(\widetilde e^\res_{pQ-Q'}\right)_{Q'\in\Delta_k^\pm(I,P_0), Q'\in \Delta_k^\pm(I,\eta(P_0))}\\
	 	=\prod_{P_0\in \Delta^-}\prod_{I \subseteq I(P_0)}\det\left(\widetilde e^\res_{P_0-\eta(P_0)+pQ-Q'}\right)_{Q,Q'\in\Delta_k^\pm(I,P_0)}.\qedhere
	 		\end{multline*}
	\end{proof}

%\begin{notation}
%	for every $P=\sum\limits_{i=1}^nz_i\mathbf{V}_i\in \Lambda_\Delta$, we put
%	$$\widetilde{e}_P^\res(T)=\sum_{j=w(P)}^{\sum_{i=1}^{n}z_i}\widetilde{e}_{P,i}^\res T^i.$$
%\end{notation}
%\begin{lemma}
%	$$\widetilde{v}^\res_{k,h(\Delta^\Int_k)}/(\pm)=\prod_{P_0\in \Delta^{\Int}}\sum_{\tau\in \Iso(S_{\Delta,k-1})}\prod_{Q\in S_{\Delta,k-1}} \widetilde{e}_{pP_0-\eta(P_0)+pQ-\tau (Q),w(pP_0-\eta(P_0)+pQ-\tau (Q))}^{\res}$$
%\end{lemma}

 \begin{definition}\label{definition of partial order}
 		\hfill
	\begin{enumerate}
		\item 	
		The \emph{partial degree} of  a monomial $\prod\limits_{i=1}^n \zeta_i^{t_i}$, denoted by $\Deg(\prod\limits_{i=1}^n \zeta_i^{t_i} )$, is defined to be the vector $(t_n,t_{n-1},\dots,t_1)\in \ZZ_{\geq 0}^n,$
		where  $\ZZ_{\geq 0}^n$ equips with a reverse lexicographic order. Namely, for two vectors $\vec{v}=(v_1,v_2,\dots,v_n)$ and $\vec{u}=(u_1,u_2,\dots,u_n)$ in $\RR^n$, we call $\vec{v}$ \emph{strictly greater than} $\vec{u}$, denoted by $\vec{v}\succ\vec{u}$, if there is $1\leq k\leq n$ such that \begin{itemize}
			\item $v_i=u_i$ for every $1\leq i\leq k-1,$ and 
			\item $v_k>u_k.$
		\end{itemize}
	Note that the partial degree of every nonzero constant is just $(0,0,\dots,0)$.

%		\item 	for every two monomials  $\prod\limits_{P\in \Ver(\Delta)} \zeta_P^{n_P}$ and $\prod\limits_{P\in \Ver(\Delta)} \zeta_P^{m_P}$, we write $$\prod\limits_{P\in \Ver(\Delta)} \zeta_P^{n_P}>\prod\limits_{P\in \Ver(\Delta)} \zeta_P^{m_P}$$ if $$\Deg(\prod\limits_{P\in \Ver(\Delta)} \zeta_P^{n_P})>\Deg(\prod\limits_{P\in \{V_I}(\Delta)} \zeta_P^{m_P}).$$
		\item The \emph{partial degree} of a polynomial $g\in \Zp[\![T]\!][\zeta_1,\zeta_2,\dots,\zeta_n]$, denoted by $\Deg(g)$, is the maximal partial degree of nonzero monomials in $g$. 
		\item For a point $Q=\sum\limits_{j=1}^\ell z_j\mathbf V_{S_j}$ such that $0=z_0<z_1<\cdots<z_\ell$, the \emph{degree} of $Q$, denoted by $\Deg(Q)$, is the $n$-dimensional vector $(v_1,v_2,\dots,v_n)$ such that $$v_{n+1-\sum\limits_{j=i}^\ell \#S_j}=z_i-z_{i-1}\quad \textrm{for every }\quad1\leq i\leq \ell,$$ and all other components are zero. 
		\item The leading term of $g$, denoted by $\LD(g)$, is the sum of monomials in $g$ of the maximal partial degree. 
		 
	\end{enumerate}
\end{definition}

\begin{property}\label{property of valuation}
	\hfill
	\begin{enumerate}
		\item Every nonzero $g\in \Zp[\![T]\!][\zeta_1,\zeta_2,\dots,\zeta_n]$ satisfies
		$$\Deg(g-\LD(g))\prec \Deg(g).$$
		\item For every two polynomials $g_1,g_2\in \Zp[\![T]\!][\zeta_1,\zeta_2,\dots,\zeta_n]$, we have 
		\begin{enumerate}[label=(\Roman*)]
			\item $\Deg(g_1-g_2)\preceq \max(\Deg(g_1),\Deg(g_2)).$
			\item $\Deg(g_1g_2)=\Deg(g_1)+\Deg(g_2).$
			\item $\LD(g_1g_2)=\LD(g_1)\LD(g_2).$
			\item $\LD(g_1)=\LD(g_2)$ if and only if $\Deg(g_1)\succ\Deg(g_1-g_2). $
		\end{enumerate}
	\end{enumerate}
\end{property}

 \begin{lemma}\label{lemma valuation of P}
	Let $ \coprod\limits_{i=1}^{\ell}S_i\subseteq\{1,2,\dots,n'\}$ for some $n'\leq n$ and $Q_1=\sum\limits_{i=1}^\ell m_i\mathbf V_{S_i}\in \Lambda_{\Delta}$ such that $0=m_0< m_1<\cdots<m_{\ell}$. Then 
	\begin{itemize}
		\item[(1)] the leading term
		\begin{equation}
		\label{LDeres}\LD\Big(\widetilde{e}^\res_{{Q_1}}\Big)=\prod_{i=1}^\ell c_{m_i-m_{i-1}} \Big(\zeta_{\sum\limits_{j=i}^n\#S_j} \pi\Big)^{m_i-m_{i-1}},
		\end{equation}
		where $\{c_i\}$ is defined in \eqref{Artin-Hasse}.
		\item[(2)] $\Deg(\widetilde{e}^\res_{{Q_1}})=\Deg({Q_1}).$
	\end{itemize}
\end{lemma}

\begin{proof}	
	(1) By \eqref{ep} and Notation~\ref{Delta(I)} (3),  the polynomial  $\widetilde{e}_{Q_1}$ can be written explicitly as 
	$$	\widetilde e^\res_{{Q_1}}=\sum_{\Big\{\vec{j}\;\Big|\;\sum\limits_{\substack{S\subseteq \{1,2,\dots,n\}\\S\neq \emptyset}}j_{S}{\mathbf V_S}={Q_1}\Big\}}\Big(\prod_{\substack{S\subseteq \{1,2,\dots,n\}\\S\neq \emptyset}}c_{j_{S}} (\zeta_{\#{S}} \pi)^{j_{S}}\Big). 
	$$
	
	The monomial $$\prod_{i=1}^n c_{m_i-m_{i-1}} \Big(\zeta_{\sum\limits_{j=i}^n \#S_j} \pi\Big)^{m_i-m_{i-1}}$$
	is the unique term in $\widetilde{e}^\res_{{Q_1}}$ of the maximal degree, which completes the proof.
	
	(2) It follows directly from (1).
	%	by taking \[j_S=\begin{cases}
	%	z_{\alpha_k}-z_{\alpha_{k-1}}& \textrm{if}\ S=\{\alpha_i\;|\;1\leq i\leq k\},\\
	%	0&\textrm{otherwise.}
	%	\end{cases}\]
	%	
	%The maximal degree of $\widetilde e^\res_{Q}$ is $(z_{\alpha_1},z_{\alpha_2}-z_{\alpha_1},\dots,z_{\alpha_n}-z_{\alpha_{n-1}})$, and the term comes from taking \[j_S=\begin{cases}
	%z_{\alpha_k}-z_{\alpha_{k-1}}& \textrm{if}\ S=\{\alpha_i\;|\;1\leq i\leq k\},\\
	%0&\textrm{otherwise.}
	%\end{cases}\]
	%Namely,  
\end{proof}

\begin{notation}\label{n'}\hfill
\begin{enumerate}

\item Let $P_0\in \Delta^-$ and $I\subseteq I(P_0)$. By relabeling indices, we may assume that $I=\{n'+1,n'+2,\dots,n\}$.
Let $S_1, S_2,\dots,S_\ell$ be an ordered disjoint subsets of $\{1,2,\dots,n'\}$ such that $\coprod\limits_{i=1}^\ell S_i\subseteq\{1,2,\dots,n'\}.$  

We set  $$\Delta_k^\pm(I,P_0;S_1,\dots,S_\ell):=\Big\{\sum_{j=1}^{\ell} z_j\mathbf V_{S_j}\in \Delta^\pm_k(I,P_0) \;\Big|\;0<z_1<\cdots<z_\ell\Big\}.$$ 
\item For $1\leq \ell\leq  n'$, we set $$\mathscr S_\ell(\Delta^\pm_k(I,P_0)):=\{\Delta_k^\pm(I,P_0;S_1,\dots,S_\ell)\},$$
where $S_1,\dots,S_\ell$ runs over all ordered disjoint subsets of $\{1,2,\dots,n'\}$.
\item Let  $$\mathscr S(\Delta^\pm_k(I,P_0))=\coprod_{\ell=0}^{n'}\mathscr S_\ell(\Delta^\pm_k(I,P_0)).$$
\end{enumerate}

Clearly, we have $$\Delta^\pm_k(I,P_0)=\coprod_{\SS\in \mathscr S(\Delta^\pm_k(I,P_0))}\SS.$$
\end{notation}

We next  to show that Theorem~\ref{lemma for res} follows from Propositions~\ref{lemma section 5 (1)} and~\ref{lemma section 5 (2)} whose proofs are given later in \S\ref{section 5.1} and \S\ref{section 5.2} respectively.

\begin{proposition}\label{lemma section 5 (1)}
	For every subset $I\subseteq I(P_0)$, we have 
	\begin{multline}\label{equation section 5 (1)}
	\LD\Big( \det(\widetilde e^\res_{P_0-\eta(P_0)+pQ-Q'})_{Q,Q'\in\Delta_k^\pm(I,P_0)}\Big)\\
	=\prod_{\SS\in\mathscr S(\Delta_k^\pm(I,P_0))} 
	\det\Big(\LD(\widetilde{e}_{P_0-\eta(P_0)+pQ-Q'}^{\res})\Big)_{Q,Q'\in \SS}.
	\end{multline}
\end{proposition}

\begin{proposition}\label{lemma section 5 (2)}
	For every point $P_0\in \Delta^-$ and every subset $\SS\in \mathscr S(\Delta_k^\pm(I,P_0))$
	we have 
	\begin{equation}
	\det\Big(\LD(\widetilde{e}_{P_0-\eta(P_0)+pQ-Q'}^{\res})\Big)_{Q,Q'\in \SS}=b_{P_0,\SS} \times g_{P_0,\SS}({\underline{\zeta}})\pi^{\sum\limits_{Q\in \SS}(\lfloor pw(Q)\rfloor-\lfloor w(Q-P_0+\eta(P_0))\rfloor )},
	\end{equation}
	where $b_{P_0,\SS}$ is a $p$-adic unit in $\ZZ_p$ and $g_{P_0,\SS}({\underline{\zeta}})$ is a monomial in $\ZZ_p[\zeta_1,\zeta_2,\dots,\zeta_n]$. 
\end{proposition}

\begin{proof}[Proof of Theorem~\ref{lemma for res}(Assuming Propositions~\ref{lemma section 5 (1)} and~\ref{lemma section 5 (2)})]
	
	By Property~\ref{property of valuation}, Propositions~\ref{proposition first}, ~\ref{lemma section 5 (1)} and~\ref{lemma section 5 (2)}, we have 
	\begin{align*}\label{equation dec}
	&\LD\Big(\det(\widetilde{e}_{pQ-Q'}^{\res})_{Q, Q'\in \Delta_k^\pm}\Big)\\
	=&\prod_{P_0\in \Delta^-}\prod_{I\subseteq I(P_0)}\LD\Big(\det(\widetilde{e}_{P_0-\eta(P_0)+pQ-Q'}^{\res})_{Q,Q'\in\Delta^\pm_k(I,P_0)}\Big)\tag{Prop~\ref{proposition first}}\\
	=&\prod_{P_0\in \Delta^-}\prod_{I\subseteq I(P_0)}\prod_{\SS\in\mathscr S(\Delta_k^\pm(I,P_0))} 
	\det\Big(\LD(\widetilde{e}_{P_0-\eta(P_0)+pQ-Q'}^{\res})\Big)_{Q,Q'\in \SS}\tag{Prop~\ref{lemma section 5 (1)}}\\
	=&\prod_{P_0\in \Delta^-}\prod_{I\subseteq I(P_0)}\prod_{\SS\in\mathscr S(\Delta_k^\pm(I,P_0))} b_{P_0,\SS} \times g_{P_0,\SS}({\underline{\zeta}})\pi^{\sum\limits_{Q\in \SS}(\lfloor pw(Q)\rfloor-\lfloor w(Q-P_0+\eta(P_0))\rfloor )}\tag{Prop~\ref{lemma section 5 (2)}}\\
	=&\pi^{h(\Delta_k^{\pm})}\prod_{P_0\in \Delta^-}\prod_{I\subseteq I(P_0)}\prod_{\SS\in\mathscr S(\Delta_k^\pm(I,P_0))} b_{P_0,\SS} \times g_{P_0,\SS}({\underline{\zeta}}),
	\end{align*}
where the last equality comes from  
\begin{multline*}
h(\Delta_k^\pm)=\coprod_{P_0\in \Delta^-}\coprod_{I\subseteq I(P_0)}\coprod_{\SS\in\mathscr S(\Delta_k^\pm(I,P_0))}\SS\\=\coprod_{P_0\in \Delta^-}\coprod_{I\subseteq I(P_0)}\coprod_{\SS\in\mathscr S(\Delta_k^\pm(I,P_0))}\{Q-P_0+\eta(P_0)\;|\;Q\in\SS\}.\qedhere
\end{multline*}
\end{proof}

%\begin{proposition}
%The term $$\prod_{S\in\mathscr S(P_0)} 
%\det\Big(\LD(\widetilde{e}_{pP_0-\eta(P_0)+pP-Q}^{\res})\Big)_{P,Q\in S}$$ is of the form $b(T)g(\underline{a})$ such that
%\begin{itemize}
%	\item $b(T)\not\equiv 0\quad \textrm{mod}\quad (T^{h(\Delta^{\Int}_1)+1}, p)$.
%	\item $g(\underline{a})$ is a monomial.
%\end{itemize}
%\end{proposition}

Now we prove Propositions~\ref{lemma section 5 (1)} and~\ref{lemma section 5 (2)}.

\subsection{Proof of Proposition~\ref{lemma section 5 (1)} assuming Proposition~\ref{lemma section 5 (2)}}
\label{section 5.1}

The essence of the proof is the only terms that contribute to the left hand side of 
\eqref{equation section 5 (1)} come from the determinant of the leading terms on its right hand side.

\begin{lemma}\label{beautiful lemma}
	Let $$\coprod\limits_{j=1}^{\ell'}S'_j\subseteq \coprod\limits_{i=1}^{\ell}S_i\subseteq\{1,2,\dots,n'\}.$$ For  every two points $Q=\sum\limits_{i=1}^\ell z_i\mathbf V_{S_i}$ and $Q'=\sum\limits_{j=1}^{\ell' }z'_j\mathbf V_{S'_j}$ in $\Lambda_{\Delta}$ such that \begin{itemize}
		\item[(1)] $0< z_1< z_2<\cdots< z_\ell$ and $0< z'_1< z'_2<\cdots< z'_{\ell'}$;
		\item[(2)]  $z_{i}-z_{i-1}>z'_{\ell'}-z'_1$ for every $1\leq i\leq \ell$,
	\end{itemize} 
	then we have \begin{equation}\label{eq}
	\Deg(Q-Q')\preceq\Deg(Q)-\Deg(Q')
	\end{equation}
	with the equality if and only if
	for every $1\leq i\leq \ell$, there is $1\leq j_i\leq \ell'$ such that 
	\begin{itemize}
		\item[(I)] 	$S_i\subseteq S'_{j_i}$, and
		\item[(II)] $j_{i_1}\leq j_{i_2}$ for every $1\leq i_1 < i_2\leq \ell$.
	\end{itemize}
\end{lemma}
\begin{proof}
When $\ell=1$, if $\ell'\leq 1$,  it is easy to show that $$\Deg(Q-Q')=\Deg(Q)-\Deg(Q').$$

Now we assume that $\ell'>1.$ 
By condition (2) and Definition~\ref{definition of partial order}(3), we obtain 
\begin{equation}\label{5.13}
\Deg(Q-Q')=(\overbrace{0,\dots,0}^{n-\sum\limits_{j=1}^\ell \#S_j},(z_1-z'_{\ell'}),\overbrace{ \cdots }^{\sum\limits_{j=1}^\ell \#S_j-1})
\end{equation}
and
\begin{equation}\label{5.14}
\Deg(Q)-\Deg(Q')=(\overbrace{0,\dots,0}^{n-\sum\limits_{j=1}^\ell \#S_j},(z_1-z_1'),\overbrace{\cdots}^{\sum\limits_{j=1}^\ell \#S_j-1}),
\end{equation}
which directly imply $\Deg(Q-Q')\prec\Deg(Q)-\Deg(Q')$.
	
Assume \eqref{eq} and its equality condition hold for $\ell>1$; we will prove them for $\ell+ 1$.
	
We set $j':=\max\{j\;|\; S_1\cap S'_j\neq\emptyset\}$. By condition (2) and Definition~\ref{definition of partial order} (3) again, we have
	\begin{equation}
	\Deg(Q-Q')=(\overbrace{0,\dots,0}^{n-\ell-1},(z_1-z'_{j'}),\dots)
	\end{equation}
and
	\begin{equation}
	\Deg(Q)-\Deg(Q')=(\overbrace{0,\dots,0}^{n-\ell-1},(z_1-z_1'),\dots).
	\end{equation}

	If $j'>1$, then $$\Deg(Q-Q')\prec\Deg(Q)-\Deg(Q').$$
	Then we show that in this case (I) and (II) cannot be both satisfied. Otherwise, from (I), we have $S_1\subseteq S_{j'}$. Since $j'>1$, there exists $1< i\leq \ell$ such that $S_i\subseteq S'_{1}$. Therefore, we obtain $j_1=j'>1=j_i$, which contradicts to (II).

		If $j'=1$, we have $S_1\subseteq S_1'$, hence $$\coprod\limits_{j=2}^{\ell'}S'_j\subseteq \coprod\limits_{i=2}^{\ell}S_i\subseteq\{1,2,\dots,n'\}.$$ We set $$Q_0:=\sum_{i=2}^{\ell+1} (z_i-z_1)\mathbf{V}_{S_i}\ \textrm{and}\ Q_0':=\sum_{j=2}^{\ell'} (z'_j-z'_1)\mathbf{V}_{S_j'}.$$ 
		It is easy to check that  $Q_0$ and $Q_0'$ also satisfy the conditions (1) and (2). 
		By condition (2) and Lemma~\ref{lemma valuation of P}, we get
	\[\Deg(Q)-\Deg(Q')-\Deg(Q-Q')=\Deg(Q_0)-\Deg(Q_0')-\Deg(Q_0-Q_0').\]
Therefore, 
	we are left to show that this result holds for $Q_0$ and $Q'_0$, which follows directly from the induction.
\end{proof}

\begin{lemma}\label{lemma1}
Let $\coprod\limits_{i=1}^{\ell}S_i\subseteq\{1,2,\dots,n'\}$ and $Q=\sum\limits_{i=1}^\ell z_i\mathbf{\mathbf{V}}_{S_i} \in \Cone(\Delta)$ such that $0<z_1<\cdots<z_\ell$. The point $(p-1)Q-\eta(Q\%)$ is a linear combination of $\{\mathbf{V}_{S_i}\}_{i=1}^\ell$ with integer coefficients, i.e. $$(p-1)Q-\eta(Q\%)
=\sum\limits_{i=1}^\ell t_i\mathbf{\mathbf{V}}_{S_i},$$ 
and  $t_{i}-t_{i-1}>n+2$ for every $1\leq i\leq \ell$, where $t_0=0$. 
\end{lemma}

\begin{proof}
	 Taking $t_{i}=\lfloor pz_i\rfloor-z_i$ we prove  $$(p-1)Q-\eta(Q\%)=\sum\limits_{i=1}^\ell t_i\mathbf{\mathbf{V}}_{S_i}.$$
	 Combined with our  assumption that $p>D(n+4)$, this implies that 
	\begin{equation*}\label{estimate}
	t_{i+1}-t_i\geq (p-1)(z_{i+1}-z_{i})-1\geq \frac{p-1}{D}-1>n+2\geq 0
	\end{equation*}
	for every $1\leq i\leq \ell$.
\end{proof} 

\begin{lemma}\label{lemma2}
	Let $1\leq k\leq n+2$, $\coprod\limits_{i=1}^{\ell}S_i\subseteq\{1,2,\dots,n'\},$ and $\coprod\limits_{j=1}^{\ell'}S'_j\subseteq\{1,2,\dots,n'\}.$ For every pair of points $Q\in \Delta_k^\pm(I,P_0;S_1,\dots,S_\ell)$ and $Q'\in \Delta_k^\pm(I,P_0;S'_1,S'_2,\dots,S'_{\ell'})$, we have \begin{equation*}\label{eqQ1}
	\Deg(pQ-\eta(P_0)+P_0-Q')\preceq\Deg(pQ-\eta(P_0)+P_0)-\Deg(Q'),
	\end{equation*}
with the equality 
	if and only if $S_i$'s and $S_j'$'s satisfy that
	\begin{itemize}
		\item for every $1\leq i\leq \ell'$ there exists $1\leq j_i\leq \ell$ such that $S_{i}
		\subseteq S'_{j_i}$, and
		\item  for every $1\leq i_1\leq \ell'$ and $1\leq i_2\leq \ell'$ if $i_1<i_2$, then $j_{i_1}\leq j_{i_2}$.
	\end{itemize}

\end{lemma}
\begin{proof}
	Let $$Q=\sum_{i=1}^{\ell}z_i \mathbf{V}_{S_i}, \quad pQ-\eta(P_0)+P_0
	=\sum_{i=1}^{\ell}t_i \mathbf{V}_{S_i}\quad  \textrm{and}\quad  Q'=\sum_{j=1}^{\ell'}z'_{j} \mathbf{V}_{S_j'}$$
	such that $0<z_1<\cdots<z_\ell$ and  $0<z'_1<\cdots<z'_{\ell'}.$
	
	 By Lemma~\ref{lemma1}, we have 
	\begin{equation}\label{estimate1}
	t_{i}-t_{i-1}>n+2\geq z'_{\ell'}-z'_1,
	\end{equation}
	for every $1\leq i\leq \ell$.
	Then it follows directly from Lemma~\ref{beautiful lemma}.
%	for every $1\leq i\leq \ell-1$, if $t_{i}=t_{i+1}$, then $m_{i}=m_{i+1}$
%	
%	\Ren{finish the proof}
%	By Lemma~\ref{beautiful lemma}, it is easy to check that the equality in \eqref{estimate1} holds if and only if $Q_1$ and $Q_2$ are in the same $S\in \mathscr S(P_0)$.
\end{proof}
\begin{notation}
	We set 
	$$K:=\min\{i\;\big|\;\mathscr S_i(\Delta_k^\pm(I,P_0))\neq \emptyset\}.$$
\end{notation}
\begin{corollary}\label{corollary2}
	Let $\SS\in \mathscr S_\ell(\Delta_k^\pm(I,P_0))$. If $Q\in \SS$ and $Q'\in \Delta_k(I, P_0)$ satisfy \begin{equation}
	\Deg(pQ-\eta(P_0)+P_0-Q')=\Deg(pQ-\eta(P_0)+P_0)-\Deg(Q'),
	\end{equation}
	then
	\begin{equation}
	Q'\in \coprod_{j=K}^\ell \coprod_{\SS'\in\mathscr S_j(\Delta_k^\pm(I,P_0))}\SS'. 
	\end{equation}
	Moreover, if $$Q'\in \coprod_{\SS'\in\mathscr S_\ell(\Delta_k^\pm(I,P_0))}\SS',$$ then $Q'\in \SS$.
\end{corollary}
\begin{proof}
	It follows directly from Lemma~\ref{lemma2}.
\end{proof}
\begin{proposition}\label{pro2}
	For a permutation $\tau \in \Iso(\Delta_k^\pm(I,P_0)),$ the equality 
	\begin{equation}\label{equation for tau}
	\Deg(pQ-\eta(P_0)+P_0-\tau(Q))=\Deg(pQ-\eta(P_0)+P_0)-\Deg(\tau(Q))
	\end{equation} holds for every $Q\in \Delta_k^\pm(I,P_0)$ if and only if 
	$$\tau(\SS)=\SS$$
	for every $\SS\in \mathscr S(\Delta_k^\pm(I,P_0))$.
\end{proposition}
\begin{proof}
	The ``if part'' follows from Lemma~\ref{lemma2}.
	
Now we prove the ``only if part''.
Assume $\tau$ is a permutation of  $\mathscr S(\Delta_k^\pm(I,P_0))$ such that every $Q\in  \Delta_k^\pm(I,P_0)$ satisfies \eqref{equation for tau}.
Combining two statements in 
Corollary~\ref{corollary2} gives that $\tau(\SS_K)=\SS_K$.

Assume that any $K\leq \ell'<\ell$ and any $\SS_{\ell'} \in  \mathscr S_{\ell'}(\Delta_k^\pm(I,P_0))$ satisfies $\tau(\SS_{\ell'})=\SS_{\ell'}$. 
Let $\SS_{\ell}$ be any subset of $\Delta_k^\pm(I,P_0)$ in $\mathscr S_{\ell}(\Delta_k^\pm(I,P_0))$. Combining the induction assumption with Corollary~\ref{corollary2}, we know $$\tau(\SS_{\ell})\subseteq \coprod_{\SS'\in\mathscr S_{\ell}(\Delta_k^\pm(I,P_0))}\SS', $$and hence $\tau(\SS_{\ell})=\SS_{\ell}$.
\end{proof}
	
\begin{notation}
	Let $$\Iso^{\sp}(\Delta_k^\pm(I,P_0)):=\big\{\tau\in \Iso(\Delta_k^\pm(I,P_0))\;\big|\;\tau(\SS)=\SS\ \textrm{for every}\ \SS\in \mathscr S(\Delta_k^\pm(I,P_0))\big\}.$$
\end{notation}
\begin{lemma}\label{lemma if and only if belones to S}
	Let $\tau$ be a permutation of $\Delta_k^\pm(I,P_0))$. 
We have
$$\Deg(\prod_{Q\in \Delta_k^\pm(I,P_0)}\widetilde{e}_{pQ-\eta(P_0)+P_0-\tau(Q)}^{\res})=\max_{\tau'\in \Iso(\Delta_k^\pm(I,P_0))}\Big(\Deg(\prod_{Q\in \Delta_k^\pm(I,P_0)}\widetilde{e}_{pQ-\eta(P_0)+P_0-\tau'(Q)}^{\res})\Big)$$
		if and only if
		$\tau\in \Iso^{\sp}(\Delta_k^\pm(I,P_0)).$
%		\item  	for every $S\in \mathscr S(P_0)$ and $\tau_1\in \Iso(S)$,
%		we have
%		$$\LD(\prod_{P\in S}\widetilde{e}_{pP_0-\eta(P_0)+pP-\tau_1(P)}^{\res})
%		=??$$

\end{lemma}
\begin{proof}
	By Property~\ref{property of valuation} (2) II, Lemma~\ref{lemma valuation of P} (2), and Lemma~\ref{lemma2}, we know that 
	\begin{align*}
	&\Deg\left(\prod_{Q\in \Delta_k^\pm(I,P_0)}\widetilde{e}_{pQ-\eta(P_0)+P_0-\tau(Q)}^{\res}\right)\\
	=&\sum_{Q\in \Delta_k^\pm(I,P_0)}\Deg\left(\widetilde{e}_{pQ-\eta(P_0)+P_0-\tau(Q)}^{\res}\right)	\tag{Property~\ref{property of valuation}(2)(II)}\\
	=&\sum_{Q\in \Delta_k^\pm(I,P_0)}\Deg\left(pQ-\eta(P_0)+P_0-\tau(Q) \tag{Lemma~\ref{lemma valuation of P}(2)}\right)
	\\\preceq& \sum_{Q\in \Delta_k^\pm(I,P_0)}\left(\Deg\Big(pQ-\eta(P_0)+P_0\Big)-\Deg\Big(\tau(Q)\Big)\tag{Lemma~\ref{lemma2}}\right)
	\\=&\sum_{Q\in \Delta_k^\pm(I,P_0)}\left(\Deg\Big(pQ-\eta(P_0)+P_0\Big)-\Deg(Q)\right).\\
	\end{align*}
	By Proposition~\ref{pro2}, the equality holds in the above inequality if and only if  \[\tau\in \Iso^{\sp}(\Delta_k^\pm(I,P_0)).\qedhere\]
\end{proof}

\begin{lemma}\label{lemma leading term}
	Let $J$ be an index set and $\{g_j\;|\;j\in J\}$ be a set of polynomials in $\Zp[\![T]\!][\zeta_1,\zeta_2,\dots,\zeta_n]$. Denote by $J_0$ the subset of $J$ such that for every $j\in J$, we have
	$$\Deg(g_j)=\max_{i\in J} \Big(\Deg(g_i)\Big)\quad \textrm{if and only if}\quad j\in J_0.$$
	
	Then we have
	either 
	\begin{enumerate}
		\item $\sum\limits_{j\in J_0} \LD(g_j)=\LD(\sum\limits_{j\in J} g_j)$ or
		\item $\sum\limits_{j\in J_0} \LD(g_j)=0$.
	\end{enumerate}  
\end{lemma}
\begin{proof}
	It is straightforward.
%	If $\sum\limits_{j\in J_0} \LD(g_j)\neq0$, then 
%	$$\Deg\Big(\sum\limits_{j\in J_0} \LD(g_j)\Big)=\max_{i\in J} \Big(\Deg(g_i)\Big).$$
%	On the other hand, it is easy to check $$\Deg\Big(\sum\limits_{j\in J} g_j-\sum\limits_{j\in J_0} \LD(g_j)\Big)<\max_{i\in J} \Big(\Deg(g_i)\Big).$$
%	Therefore, by Property~\ref{property of valuation} (IV), we obtain \[\LD(\sum\limits_{j\in J} g_j)=\LD\Big(\sum\limits_{j\in J_0} \LD(g_j)\Big)=\sum\limits_{j\in J_0} \LD(g_j).\qedhere\]
\end{proof}

We prove Proposition~\ref{lemma section 5 (1)} assuming Proposition~\ref{lemma section 5 (2)}.
\begin{proof}[Proof of Proposition~\ref{lemma section 5 (1)}]
	By Property~\ref{property of valuation}(2)(III) and the definition of $\Iso^{\sp}(\Delta_k^\pm(I,P_0))$, we have
	\begin{multline}
	\sum_{\tau\in \Iso^{\sp}(\Delta_k^\pm(I,P_0))}\LD\Big(\sgn(\tau)\prod_{P\in \Delta_k^\pm(I,P_0)}\widetilde{e}_{P_0-\eta(P_0)+pQ-\tau(Q)}^{\res}\Big)\\=\prod_{\SS\in\mathscr S(\Delta_k^\pm(I,P_0))} 
	\det\Big(\LD(\widetilde{e}_{P_0-\eta(P_0)+pQ-Q'}^{\res})\Big)_{Q,Q'\in \SS}.
	\end{multline}
	Assuming Proposition~\ref{lemma section 5 (2)}, we know that 
	\begin{equation}\label{equation neq 0}
\prod_{\SS\in\mathscr S(\Delta_k^\pm(I,P_0))} 
\det\Big(\LD(\widetilde{e}_{P_0-\eta(P_0)+pQ-Q'}^{\res})\Big)_{Q,Q'\in \SS}\neq 0.
	\end{equation}
	Combining \eqref{equation neq 0} with Lemmas~\ref{lemma if and only if belones to S} and ~\ref{lemma leading term}, we obtain \eqref{equation section 5 (1)}.
\end{proof}

%The following lemma shows that we only need to consider the leading terms of each $\widetilde e^\res_{pP-Q}$ in \eqref{equation dec}. It will simplifies the computations since each $\LD(\widetilde e^\res_{pP-Q})$ is monomial. 
%
%
%\begin{lemma}
%	We have
%	$$\LD\Big( \det(\widetilde{e}_{pP_0-\eta(P_0)+pP-Q}^{\res})_{P,Q\in S_{\Delta,k-1}}\Big)=\LD\Big(\det(\LD(\widetilde{e}_{pP_0-\eta(P_0)+pP-Q}^{\res}))_{P,Q\in S_{\Delta,k-1}}\Big).$$
%\end{lemma}
%
%\begin{proof}
%			we have 
%	\begin{enumerate} 
%		\item $\Deg\Big(\det(\widetilde{e}_{pP_0-\eta(P_0)+pP-Q}^{\res})_{P,Q\in S_{\Delta,k-1}}-\det(\LD(\widetilde{e}_{pP_0-\eta(P_0)+pP-Q}^{\res}))_{P,Q\in S_{\Delta,k-1}}\Big)\\
%		<\prod\limits_{P\in S_{\Delta,k-1}}\Deg(\widetilde{e}_{pP_0-\eta(P_0)+pP-P}),$
%		and
%		\item		$\Deg\Big(\det(\LD(\widetilde{e}_{pP_0-\eta(P_0)+pP-Q}^{\res}))_{P,Q\in S_{\Delta,k-1}}\Big)=\prod\limits_{P\in S_{\Delta,k-1}}\Deg(\widetilde{e}_{pP_0-\eta(P_0)+pP-P}).$
%	\end{enumerate}
%Combining the equalities above with Property~\ref{property of valuation} (2) (IV), we complete the proof.
%\end{proof}

\subsection{Proof of Proposition~\ref{lemma section 5 (2)}}\label{section 5.2}

\begin{notation}

We set $$Y_{k}(S_1,\dots,S_\ell):=\Big\{\sum_{i=1}^{\ell} m_i\mathbf V_{S_i}\;\Big|\;m_i\in \ZZ\quad \textrm{and}\quad 0\leq m_1\leq m_2\leq\cdots\leq m_{\ell}\leq k\Big\}.$$ 
\end{notation}

As in Notation~\ref{n'}, we assume that $I=\{n'+1,n'+2,\dots,n\}$.
\begin{lemma}\label{lemma P1}
	For every nonempty set $\Delta_k^\pm(I,P_0;S_1,\dots,S_\ell)$, there exist $$Q_\min=\sum_{i=1}^\ell z_{\min,i}\mathbf V_{S_i}\in \Delta_k^\pm(I,P_0;S_1,\dots,S_\ell),$$ where $0<z_{\min,i}-z_{\min,i-1}\leq1$ for all $1\leq i\leq \ell$,
	and integers $K^\pm$ such that 
	\begin{equation}\label{equation S}
\Delta_k^\pm(I,P_0;S_1,\dots,S_\ell)=Q_\min+Y_{K^\pm}(S_1,\dots,S_\ell).
	\end{equation}

\end{lemma}
\begin{proof}
Let $Q_\min$ denote the point in  $\Delta_k^\pm(I,P_0;S_1,\dots,S_\ell)$ with the minimal degree. Now we show that $0<z_{\min,i}-z_{\min,i-1}\leq 1$ for all $1\leq i\leq \ell$. Suppose that it is false, then there exists an integer $j$ such that  $z_{\min,j}-z_{\min,j-1}> 1$. It is easy to check that
$\sum\limits_{i=1}^\ell z_{\min,i}\mathbf V_{S_i}-\mathbf V_{S_j}$ is a point in $\Delta_k^\pm(I,P_0;S_1,\dots,S_\ell)$
of a smaller degree than $Q_\min$, a contradiction.

The rest of this lemma is obvious. We show it by an example.
\end{proof}

\begin{example}
	When $\Delta $ is a cube generated by $(3,0,0), (0,3,0),(0,0,3)$, $p=29$, and $P_0=(1,0,0)$.
	
	(1) For $\Delta_3^\pm(\{2\},P_0;\{1\}, \{3\})$, it is easy to show that 
	\begin{itemize}
		\item $\Delta_3^-(\{2\},P_0;\{1\},\{3\})=\{(1,0,3),(1,0,6),(4,0,6)\}$ and
		\item $\Delta_3^+(\{2\},P_0;\{1\},\{3\})=\{(1,0,3),(1,0,6),(4,0,6),(1,0,9),(4,0,9),(7,0,9)\},$
		\end{itemize}
	hence $Q_\min=(1,0,3)$.
	
	It is easy to check that 
	 $$Q_\min+Y_{1}(\{1\},\{3\})=\Delta_3^-(\{2\},P_0;\{1\},\{3\})$$
	and 
	\[
	Q_\min+Y_{2}(\{1\},\{3\})=\Delta_3^+(\{2\},P_0;\{1\},\{3\}).\]

	(2)  For $\Delta_3^\pm(\{2\},P_0;\{3\},\{1\}),$ we get
	$$\Delta_3^\pm(\{2\},P_0;\{3\},\{1\})=\{(4,0,3),(7,0,3),(7,0,6)\},$$
	hence $ Q_\min=(4,0,3)$.
	
	It is easy to check that 
	 $$Q_\min+Y_{1}(\{3\},\{1\})=\Delta_3^\pm(\{2\},P_0;\{3\},\{1\}).$$
\end{example}

Recall that in Notation~\ref{c}, we denote by $c_i\in \Zp$ the coefficients of $\pi^i$ in $E(\pi)$ .

\begin{notation}\label{defgamma}
	Let $Q_1=\sum\limits_{i=1}^\ell m_i\mathbf V_{S_i}\in \Lambda_{\Delta}$ such that $0\leq m_1\leq \cdots\leq m_\ell$. 
		 We put \begin{equation}\label{gamma}
		 \gamma(Q_1):=\pi^{m_\ell}\prod_{i=1}^{\ell}\Big(\zeta_{\sum\limits_{j=i}^{\ell} \#S_{j}(Q_1)}\Big)^{m_i-m_{i-1}}.
		 \end{equation}
\end{notation}
\begin{remark}\label{samedef}
Although the method of writing $Q_1$ as a sum like that is not unique, $\gamma(Q_1)$ is well-defined.
\end{remark}

%\begin{proposition}
%	for every $P\in \Lambda_{\Delta}$ if $$n+2<m_{i+1}(P)-m_i(P)<p-(n+2)$$ for every $0\leq i\leq \ell(P)-1$,
%	we have 
%	\begin{equation}
%	\det\Big(\LD(\widetilde{e}_{P+pQ-Q'}^{\res})\Big)_{Q,Q'\in Y_k(P)}=* \times \prod_{Q\in Y_k(P)} (\gamma_P(P+pQ)-\gamma_P(Q)),
%	\end{equation}
%where $*$ is a $p$-adic unit in $\ZZ_p$.
%\end{proposition}
%We give its proof after several lemmas.
\begin{lemma}\label{lemma M1}
	Let $k\geq 1$ be an integer, and let $Q=\sum\limits_{i=1}^\ell z_i\mathbf V_{S_i}$ such that $$z_i-z_{i-1}> k\quad \textrm{for every}\quad 1\leq i\leq \ell.$$
	For every $Q_1=\sum\limits_{i=1}^\ell m_i\mathbf V_{S_i}\in Y_k(S_1,\dots,S_{\ell})$ and every  $Q_2=\sum\limits_{i=1}^\ell m'_i\mathbf V_{S_i}\in Y_k(S_1,\dots,S_{\ell})$, the leading term $$\LD(\widetilde{e}_{Q+pQ_1-Q_2}^{\res})=\frac{\gamma(Q)\gamma(pQ_{1})}{\gamma(Q_{2})}\prod_{i=1}^{\ell}c_{z_i+pm_i-m'_i}.$$
\end{lemma}
\begin{proof}
	By Lemma~\ref{lemma valuation of P}, we have $$\LD(\widetilde{e}_{Q+pQ_1-Q_2}^{\res})=\gamma(Q+pQ_{1}-Q_{2})\prod_{i=1}^{\ell}c_{z_i+pm_i-m'_i}.$$
	
	By Notation~\ref{defgamma}, we know that $$\gamma(Q+pQ_{1}-Q_{2})=\frac{\gamma(Q)\gamma(pQ_{1})}{\gamma(Q_{2})}.$$
	Combining them, we complete the proof.
\end{proof}
\begin{notation} \hfill
		\begin{enumerate}
			\item 
			For a vector $\vec{w}=(w_1,w_2,\dots,w_\ell) \in \ZZ^\ell$ such that $w_0<w_1<\cdots<w_\ell$, we denote $$\xi(\vec{w}):=\prod_{i=1}^{\ell-1}c_{w_{i+1}-w_i}.$$
			
			\item Let
			$$V_{\ell,k}:=\big\{(z_1,z_2,\dots,z_\ell)\in \ZZ^\ell\;\big|\;0\leq z_{1}\leq \cdots\leq z_{\ell}\leq k\big\}.$$ 
			
			Rearranging the vectors in $V_{\ell,k}$ with 
			increasing partial order as defined in Definition~\ref{definition of partial order}(1), we obtain a sequence of vectors as
			$\vec v_1,\vec v_2,\dots,\vec v_{\tbinom{\ell+k}{k}}$. For a vector $\vec{w}=(w_1,w_2,\dots,w_\ell) \in \ZZ^\ell$ such that $k<w_i-w_{i-1}$ for every $1\leq i\leq \ell$, we set $$M(\vec{w},k):=\Big(\xi(\vec{w}+p\vec v_i-\vec v_j)\Big)_{1\leq i,j\leq \tbinom{\ell+k}{k}}.$$
			
		\end{enumerate}
	\end{notation}
	\begin{proposition}\label{lemma M}
			Let $k\in \ZZ_{>0}$ and $Q=\sum\limits_{i=1}^\ell z_i\mathbf V_{S_i}$ such that $0=z_0<z_1<\dots<z_\ell$. If $$z_i-z_{i-1}> k\quad \textrm{for every}\quad 1\leq i\leq \ell,$$
			then
\begin{multline}
\det\Big(\LD(\widetilde{e}_{Q+pQ_1-Q_2}^{\res})\Big)_{Q_1,Q_2\in Y_k(S_1,\dots,S_{\ell})}
\\=\det \Big(M\Big((z_1,\dots,z_\ell),k\Big)\Big) \times \prod_{Q_1\in Y_k(S_1,\dots,S_{\ell})} \gamma\big(Q+(p-1)Q_1\big).
\end{multline}
\end{proposition}

\begin{proof}
			By Lemma~\ref{lemma M1}, we simplify  $\det\Big(\LD(\widetilde{e}_{Q+pQ_1-Q_2}^{\res})\Big)_{Q_1,Q_2\in Y_k(S_1,\dots,S_{\ell})}$ as follows:
	\begin{align*}
	&\det\Big(\LD(\widetilde{e}_{Q+pQ_1-Q_2}^{\res})\Big)_{Q_1,Q_2\in Y_k(S_1,\dots,S_{\ell})}\\
	=&\det\Bigg(\diag\Big(\gamma(Q_1)^{-1}\Big)_{Q_1\in Y_k(S_1,\dots,S_{\ell})}\Big(\LD(\widetilde{e}_{Q+pQ_1-Q_2}^{\res})\Big)_{Q_1,Q_2\in Y_k(S_1,\dots,S_{\ell})}\diag\Big(\gamma(Q_2)\Big)_{Q_2\in Y_k(S_1,\dots,S_{\ell})}\Bigg)\\
	=&\det\Big(\LD\Big(\widetilde{e}_{Q+pQ_1-Q_2}^{\res}\Big)\gamma(Q_2)\gamma^{-1}(Q_1)\Big)_{Q_1,Q_2\in Y_k(S_1,\dots,S_{\ell})}\\
	=&\det \Bigg(  \diag\Big(\gamma(Q+pQ_1-Q_1)\Big)_{Q_1\in Y_k(S_1,\dots,S_{\ell})}M\Big((z_1,\dots,z_\ell),k\Big)\Bigg)\tag{Lemma~\ref{lemma M1}}\\
	=&\det \Big(M\Big((z_1,\dots,z_\ell),k\Big)\Big)\times \prod_{Q_1\in Y_k(S_1,\dots,S_{\ell})} \Big(\gamma(Q+pQ_1-Q_1)\Big),
\end{align*}
where $\diag\Big(\gamma(Q_1)\Big)_{Q_1\in Y_k(S_1,\dots,S_{\ell})}$ is a diagonal matrix whose rows and columns are indexed by the points in $Y_k(S_1,\dots,S_{\ell})$.
\end{proof}

\begin{lemma}\label{lemma M=}
	Let $\vec{w}=(w_1,w_2,\dots,w_\ell)$ be an $\ell$-dimensional vector. If \begin{equation}\label{condition}
	k\leq w_{i}-w_{i-1}\leq p-k
	\end{equation} 
	for every $1\leq i\leq \ell$, then $$\det\Big(M(\vec{w},k)\Big)\not\equiv 0 \pmod p.$$
\end{lemma}
\begin{proof}
	We prove it by induction. 
	When $\ell=1$, we have $\vec w=(w_1)$. For every  $k\geq 1$ we write the determinant explicitly as
	\begin{equation}
	\label{E:explicit N}
	\det\Big(M(\vec w,k)\Big) =\det\begin{pmatrix} 
	c_{w_1}&c_{w_1+p}&\cdots&c_{w_1+kp}\\
	c_{w_1-1}& c_{w_1+p-1} & \cdots & c_{w_1+kp-1}\\
	c_{w_1-2} & c_{w_1+p-2} & \cdots & c_{w_1+kp-2}\\
	\vdots & \vdots & \ddots & \vdots\\
	c_{w_1-k} & c_{w_1+p-k} & \cdots & c_{w_1+kp-k}
	\end{pmatrix}.
	\end{equation}
	By \cite{ren-wan-xiao-yu}[Lemma~5.2] and \eqref{condition}, the congruence relation \begin{equation}\label{con}
	(w_1-j)c_{w_1+ip-j}-c_{w_1+ip-j-1}\equiv c_{w_1+(i-1)p-j}\pmod p
	\end{equation}
	holds for every $0\leq i\leq k$ and $0\leq j\leq k-1$, where $c_s=0$ for all $s<0$. 
	
Therefore, by \eqref{con}, the equality \eqref{E:explicit N} can be simplified as
\begin{equation}\label{aaa}
\begin{split}
&\det\Big(M(\vec w,k)\Big)\\
=&\det\begin{pmatrix} 
c_{w_1}&c_{w_1+p}&\cdots&c_{w_1+kp}\\
c_{w_1-1}-w_1c_{w_1}& c_{w_1+p-1}- w_1c_{w_1+p}& \cdots &  c_{w_1+kp-1}-w_1c_{w_1+kp}\\
c_{w_1-2}-(w_1-1)c_{w_1-1} & c_{w_1+p-2}-(w_1-1)c_{w_1+p-1} & \cdots & c_{w_1+kp-2}-(w_1-1)c_{w_1+kp-1}\\
\vdots & \vdots & \ddots & \vdots\\
\scriptstyle c_{w_1-k}-(w_1-k+1)c_{w_1-k+1} & \scriptstyle c_{w_1+p-k}-(w_1-k+1)c_{w_1+p-k+1} & \cdots & \scriptstyle c_{w_1+kp-k}-(w_1-k+1)c_{w_1+kp-k+1}
\end{pmatrix}\\
\equiv&\det\begin{pmatrix} 
c_{w_1}&c_{w_1+p}&\cdots&-c_{w_1+kp}\\
0& -c_{w_1} & \cdots & -c_{w_1+(k-1)p}\\
0 & -c_{w_1-1} & \cdots & -c_{w_1+(k-1)p-1}\\
\vdots & \vdots & \ddots & \vdots\\
0 & -c_{w_1-k+1} & \cdots &- c_{w_1+(k-1)p-k+1}
\end{pmatrix}\equiv(-1)^k c_{w_1}\det\Big(M(\vec w,k-1)\Big)\pmod p.
\end{split}
\end{equation}

Since $k-1\leq w_{i+1}-w_i\leq p-k+1$ for every $0\leq i\leq k-1$, by simply taking induction on $k$, we know that 
\begin{equation}\label{equation of determine}
\det\Big(M(\vec w,k)\Big)\equiv (-1)^{\frac{k(k+1)}{2}}c_{w_1}^{k-1}\pmod p.
\end{equation}

Condition~\eqref{condition} shows
\begin{equation}\label{equation p-unit}
c_{w_1}=\frac{1}{w_1!}\not\equiv 0\pmod p.
\end{equation}

Combining \eqref{equation p-unit} and \eqref{equation of determine}, we complete the proof of the case when $\ell=1$.

For $\ell=2$, by definition, we have 
\[M(\vec w,k)=\begin{pmatrix} 
c_{w_1}M_{00}&c_{w_1+p}M_{01}&\cdots&c_{w_1+kp}M_{0k}\\
c_{w_1-1}M_{10}&c_{w_1-1+p}M_{11}&\cdots&c_{w_1-1+kp}M_{1k}\\
c_{w_1-2}M_{20}&c_{w_1-2+p}M_{21}&\cdots&c_{w_1-2+kp}M_{2k}\\
\vdots & \vdots & \ddots & \vdots\\
c_{w_1-k}M_{k0}&c_{w_1-k+p}M_{k1}&\cdots&c_{w_1-k+kp}M_{kk}\\
\end{pmatrix},\]
where for every $0\leq i\leq k$ and every $0\leq j\leq k$, 
\[M_{ij}=\begin{pmatrix} 
c_{w_2-w_1}&c_{w_2-w_1+p}&\cdots&c_{w_2-w_1+jp}\\
c_{w_2-w_1-1}& c_{w_2-w_1+p-1} & \cdots & c_{w_2-w_1+jp-1}\\
c_{w_2-w_1-2} & c_{w_2-w_1+p-2} & \cdots & c_{w_2-w_1+jp-2}\\
\vdots & \vdots & \ddots & \vdots\\
c_{w_2-w_1-i} & c_{w_2-w_1+p-i} & \cdots & c_{w_2-w_1+jp-i}
\end{pmatrix}.\]

Let $A_{i}=[I_{(n-i)\times (n-i)} ,0_{(n-i)\times 1}]$ for $1\leq i\leq k$. It is easy to see that $M_{i,j}=A_{i}M_{i-1,j}$ for every	 $0\leq j\leq k$ and every $1\leq i\leq k$. Therefore, imitating the row operations used in \eqref{aaa}, we modify $M(\vec w,k)$ by ``block row operations" as follows:
\begin{equation}\label{eql2}
\begin{split}
&\begin{pmatrix} 
I_0&&&\\
\scriptstyle -w_1A_1&I_1& &\\
& & \ddots & \\
& & &I_k\\
\end{pmatrix}\begin{pmatrix} 
I_0&&&&\\
&I_1&&&\\
&\scriptstyle -(w_1-1)A_2&I_2&&\\
&  && \ddots & \\
&&&&I_k\\
\end{pmatrix}\begin{pmatrix} 
I_0&&&&\\
&I_1&&&\\
&  & \ddots && \\
&  & &I_{k-1} & \\
&&&\scriptstyle -(w_1-k+1)A_{k}&I_k\\
\end{pmatrix}M(\vec w,k)\\
=&\begin{pmatrix} 
c_{w_1}M_{00}&c_{w_1+p}M_{01}&\cdots&c_{w_1+kp}M_{0k}\\
(c_{w_1-1}-w_1c_{w_1})M_{10}&(c_{w_1-1+p}-w_1c_{w_1+p})M_{11})&\cdots&(c_{w_1-1+kp}-w_1c_{w_1+kp})M_{1k}\\
\scriptstyle (c_{w_1-2}-(w_1-1)c_{w_1-1})M_{20}&\scriptstyle (c_{w_1+p-2}-(w_1-1)c_{w_1+p-1})M_{21}&\cdots&\scriptstyle (c_{w_1+k-2}-(w_1-1)c_{w_1+k-1})M_{2k}\\
\vdots & \vdots & \ddots & \vdots\\
\scriptstyle (c_{w_1-k}-(w_1-k+1)c_{w_1-k+1})M_{k0}&\scriptstyle (c_{w_1-k+p}-(w_1-k+1)c_{w_1-k+1+p})M_{k1}&\cdots&\scriptstyle (c_{w_1-k+pk}-(w_1-k+1)c_{w_1-k+1+kp})M_{kk}\\
\end{pmatrix}\\
\equiv&\begin{pmatrix} 
c_{w_1}M_{00}&c_{w_1+p}M_{01}&\cdots&-c_{w_1+kp}M_{0k}\\
0& -c_{w_1}M_{11} & \cdots & -c_{w_1+(k-1)p}M_{1k}\\
0 & -c_{w_1-1}M_{21} & \cdots & -c_{w_1+(k-1)p-1}M_{2k}\\
\vdots & \vdots & \ddots & \vdots\\
0 & -c_{w_1-k+1}M_{k1} & \cdots &- c_{w_1+(k-1)p-k+1}M_{kk}
\end{pmatrix}\pmod p.
\end{split}
\end{equation}

It is not hard to show that \begin{equation}\label{dd}
M(\vec w,k-1)=\begin{pmatrix} 
c_{w_1} M_{11}& \cdots & c_{w_1+(k-1)p}M_{1k}\\
c_{w_1-1} M_{21}& \cdots & c_{w_1+(k-1)p-1}M_{2k}\\
\vdots & \ddots & \vdots\\
c_{w_1-k+1}M_{k1} & \cdots & c_{w_1+(k-1)(p-1)}M_{kk}
\end{pmatrix}.
\end{equation}
Combining it with \eqref{eql2} gives $$\det\Big(M(\vec w,k)\Big)\equiv
\det\begin{pmatrix} 
c_{w_1}M(\vec w',k)&\scalebox{1.5}{*}\\0& -M(\vec w,k-1)
\end{pmatrix}\pmod p,
$$
where $\vec w'=(w_1)$ and $*$ represents a $\binom{k+\ell-1}{k}\times \binom{k+\ell-1}{k-1}$ matrix. 
%where 
%\[M(\vec w,k)=\begin{pmatrix} 
%	c_{w_1}M_0&c_{w_1+p}M_1&\cdots&c_{v_1+kp}M_k\\
%	c_{v_1-1}A_1M_0& c_{v_1+p-1} A_1M_1& \cdots & c_{v_1+kp-1}A_1M_k\\
%	c_{v_1-2} A_2A_1M_0& c_{v_1+p-2} A_1A_2M_1& \cdots & c_{v_1+kp-2}A_2A_1M_k\\
%	\vdots & \vdots & \ddots & \vdots\\
%	c_{v_1-k} A_kA_{k-1}\cdots A_1M_0& c_{v_1+p-k}A_kA_{k-1}\cdots A_1M_1 & \cdots & c_{v_1+kp-k}A_kA_{k-1}\cdots A_1M_k
%\end{pmatrix}.\]
%$M(\vec w',k)$
%
%By writing down the explicit matrix $$\b(\alpha(\vec w+pQ_i-Q_j))_{0\leq i,j\leq \binom{\ell+k}{k}}$$ and applying the similar row operations as in \eqref{aaa}, we show that 
%where $\vec w'=(v_2-v_1,v_3-v_1,\dots,v_t-v_{1})$.

Since $\vec w'$ is a one-dimensional vector, as the argument above, the determinant $$\det \Big(M(\vec w',k)\Big)\not\equiv 0\pmod p.$$
Combining it with \eqref{equation p-unit} shows that
  $$\det \Big(M(\vec w,k)\Big)\not\equiv 0\pmod p\quad \textrm{if and only if}\quad \det \Big(M(\vec w,k-1)\Big)\not\equiv 0\pmod p.$$
Since 
\[M(\vec w,0)= c_{w_{2}-w_{1}}= \frac{1}{(w_{2}-w_{1})!}\not\equiv 0\pmod p,\] 
we show that \[\det\Big(M(\vec w,k)\Big)\not\equiv 0\pmod p.\]

 Assume this statement holds for all $t\leq \ell-1$; we will prove it for $\ell$. 

We first put
\[M(\vec w,k)=\begin{pmatrix} 
c_{w_1}M_{00}&c_{w_1+p}M_{01}&\cdots&c_{w_1+kp}M_{0k}\\
c_{w_1-1}M_{10}&c_{w_1-1+p}M_{11}&\cdots&c_{w_1-1+kp}M_{1k}\\
c_{w_1-2}M_{20}&c_{w_1-2+p}M_{21}&\cdots&c_{w_1-2+kp}M_{2k}\\
\vdots & \vdots & \ddots & \vdots\\
c_{w_1-k}M_{k0}&c_{w_1-k+p}M_{k1}&\cdots&c_{w_1-k+kp}M_{kk}\\
\end{pmatrix},\]
where $M_{ij}$ is a $\binom{k-1+\ell-i}{k-1}\times\binom{k-1+\ell-j}{k-1}$ matrix; and \begin{equation}\label{M00}
M_{00}=
M(\vec w',k)
\end{equation} for $\vec w'=(w_2-w_1,w_3-w_1,\dots,w_\ell-w_{1})$.

Similar to the case $\ell=2$ there exists a set of matrices $\{A_i\}_{1\leq i\leq k}$,  such that for every  $1\leq i\leq k$ and $0\leq j\leq k$, $A_i$ is a $\binom{k-1+\ell-i}{k-1}\times\binom{k-1+\ell-i+1}{k-1}$  reduced echelon matrix and $$M_{i,j}=A_iM_{i-1,j}.$$

Similar to \eqref{eql2} and \eqref{dd}, we have
\begin{equation}
\begin{split}
&\begin{pmatrix} 
I_0&&&\\
\scriptstyle -w_1A_1&I_1& &\\
& & \ddots & \\
& & &I_k\\
\end{pmatrix}\begin{pmatrix} 
I_0&&&&\\
&I_1&&&\\
&\scriptstyle -(w_1-1)A_2&I_2&&\\
&  && \ddots & \\
&&&&I_k\\
\end{pmatrix}\begin{pmatrix} 
I_0&&&&\\
&I_1&&&\\
&  & \ddots && \\
&  & &I_{k-1} & \\
&&&\scriptstyle -(w_1-k+1)A_k&I_k\\
\end{pmatrix}M(\vec w',k)\\
\equiv&\begin{pmatrix} 
c_{w_1}M_{00}&c_{w_1+p}M_{01}&\cdots&c_{w_1+kp}M_{0k}\\
0& -c_{w_1} M_{11}& \cdots & -c_{w_1+(k-1)p}M_{1k}\\
0& -c_{w_1-1} M_{21}& \cdots & -c_{w_1+(k-1)p-1}M_{2k}\\
\vdots & \vdots & \ddots & \vdots\\
0& -c_{w_1-k+1}M_{k1} & \cdots & -c_{w_1+(k-1)(p-1)}M_{kk}
\end{pmatrix}
\\
=&\begin{pmatrix} 
c_{w_1}M(\vec w',k)&\scalebox{1.5}{*}\\0& -M(\vec w,k-1)
\end{pmatrix}\pmod p, 
\end{split}
\end{equation}
where $*$ represents a $\binom{k+\ell-1}{k}\times \binom{k+\ell-1}{k-1}$ matrix, 
and $$M(\vec w,k-1)=\begin{pmatrix} 
c_{w_1} M_{11}& \cdots & c_{w_1+(k-1)p}M_{1k}\\
c_{w_1-1} M_{21}& \cdots & c_{w_1+(k-1)p-1}M_{2k}\\
 \vdots & \ddots & \vdots\\
 c_{w_1-k+1}M_{k1} & \cdots & c_{w_1+(k-1)(p-1)}M_{kk}
\end{pmatrix}.$$

%where 
%\[M(\vec w,k)=\begin{pmatrix} 
%	c_{w_1}M_0&c_{w_1+p}M_1&\cdots&c_{v_1+kp}M_k\\
%	c_{v_1-1}A_1M_0& c_{v_1+p-1} A_1M_1& \cdots & c_{v_1+kp-1}A_1M_k\\
%	c_{v_1-2} A_2A_1M_0& c_{v_1+p-2} A_1A_2M_1& \cdots & c_{v_1+kp-2}A_2A_1M_k\\
%	\vdots & \vdots & \ddots & \vdots\\
%	c_{v_1-k} A_kA_{k-1}\cdots A_1M_0& c_{v_1+p-k}A_kA_{k-1}\cdots A_1M_1 & \cdots & c_{v_1+kp-k}A_kA_{k-1}\cdots A_1M_k
%\end{pmatrix}.\]
%$M(\vec w',k)$
%
%By writing down the explicit matrix $$\b(\alpha(\vec w+pQ_i-Q_j))_{0\leq i,j\leq \binom{\ell+k}{k}}$$ and applying the similar row operations as in \eqref{aaa}, we show that 
%where $\vec w'=(v_2-v_1,v_3-v_1,\dots,v_t-v_{1})$.

By induction and \eqref{equation p-unit}, the determinant $$\det \Big(c_{w_1}M(\vec w',k)\Big)\not\equiv 0\pmod p,$$ which implies  $$\det \Big(M(\vec w,k)\Big)\not\equiv 0\pmod p\quad \textrm{if and only if}\quad \det \Big(M(\vec w,k-1)\Big)\not\equiv 0\pmod p.$$

Since 
\[M(\vec w,0)=\prod_{i=0}^\ell c_{w_{i+1}-w_{i}}=\prod_{i=0}^\ell \frac{1}{(w_{i+1}-w_{i})!}\not\equiv 0\pmod p,\] 
we show that \[\det\Big(M(\vec w,k)\Big)\not\equiv 0\pmod p.\qedhere\]
\end{proof}

\begin{proof}[Proof of Proposition~\ref{lemma section 5 (2)}]
	Let $\SS=\Delta^\pm_{k}(P_0, I; S_1,S_2,\dots,S_\ell).$
	By Lemma~\ref{lemma P1}, there exists 
$$Q_\min=\sum_{i=1}^\ell z_{\min,i}\mathbf V_{S_i}\in \Delta_k^\pm(I,P_0;S_1,\dots,S_\ell),$$ where $0<z_{\min,i}-z_{\min,i-1}\leq1$ for all $1\leq i\leq \ell$,
	and integers $K^\pm$ such that 
\begin{equation}
\Delta_k^\pm(I,P_0;S_1,\dots,S_\ell)=Q_\min+Y_{K^\pm}(S_1,\dots,S_\ell).
\end{equation}
	It is easy to see that we can put $$Q:=(p-1)Q_\min-\eta(P_0)+P_0=\sum_{i=1}^\ell z_i\mathbf V_{S_i}\in \Lambda_{\Delta}. $$ 
	Since $P_0$ and $\eta(P_0)$ are both in $\Delta^-$ and $z_i\in \ZZ$ for every $1\leq i\leq \ell$, we know that $$\Big|z_i-z_{i-1}- (p-1)(z_{\min,i}-z_{\min,i-1})\Big|\leq1.$$
	
	Therefore, by Hypothesis~\ref{hypothesis}, 
	we get 
	\begin{equation}\label{conditionfor5.26}
	n+2< z_{i}-z_{i-1}< p-(n+2)
	\end{equation} 
	for every  $1\leq i\leq \ell.$
	
	 Since $K^\pm< n+2$, we know that $Q$ and $Y_{K^\pm}(S_1,S_2,\dots,S_\ell)$ satisfy the condition in Proposition~\ref{lemma M}, hence
	\begin{equation}
	\begin{split}
	&\det\Big(\LD(\widetilde{e}_{P_0-\eta(P_0)+pP-Q}^{\res})\Big)_{P,Q\in \SS}\\
	=&\det\Big(\LD(\widetilde{e}_{Q+pQ_1-Q_2}^{\res})\Big)_{Q_1,Q_2\in Y_{K^\pm}(S_1,\dots,S_\ell)}\\
	=&\det \Big(M\Big((z_1,\dots,z_\ell),K^\pm\Big)\Big) \times \prod_{Q_1\in Y_{K^\pm}(S_1,\dots,S_\ell)} (\gamma(Q+pQ_1-Q_1))\\
	=&\det \Big(M\Big((z_1,\dots,z_\ell),K^\pm\Big)\Big)\times g_{P_0,\SS}({\underline{\zeta}})\pi^{\sum\limits_{Q_1\in Y_{K^\pm}(S_1,\dots,S_\ell)}w(Q+pQ_1-Q_1)}\\
	=&\det \Big(M\Big((z_1,\dots,z_\ell),K^\pm\Big)\Big)\times g_{P_0,\SS}({\underline{\zeta}})\pi^{\sum\limits_{Q'\in \SS}w(P_0-\eta(P_0)+pQ'-Q')}\\
	=&\det \Big(M\Big((z_1,\dots,z_\ell),K^\pm\Big)\Big)\times g_{P_0,\SS}({\underline{\zeta}})\pi^{\sum\limits_{Q'\in \SS}\Big(\lfloor pw(Q')\rfloor -\lfloor w(\eta(P_0)-P_0+Q')\rfloor \Big)}.\\
	\end{split}
	\end{equation}
	
	By \eqref{conditionfor5.26}, the vector $(z_1,z_2,\dots,z_\ell)$ satisfies the conditions in  Lemma~\ref{lemma M=}. Therefore, we obtain $$\det \Big(M\Big((z_1,\dots,z_\ell),K^\pm\Big)\Big)\not\equiv 0 \pmod p,$$ which completes the proof.
\end{proof}

\end{document}